\documentclass[]{siamart0516}

\usepackage[english]{babel}

\usepackage[letterpaper,top=2cm,bottom=2cm,left=3cm,right=3cm,marginparwidth=1.75cm]{geometry}
\usepackage{csquotes}
\usepackage{amsmath}
\usepackage{amssymb}
\usepackage{graphicx}
\usepackage{tabularx}
\usepackage{rotating}

\usepackage{tikz}
\usepackage{mathrsfs}
\usepackage{standalone}
\usepackage{caption}
\usepackage{subcaption}
\usetikzlibrary{arrows.meta}
\tikzset{%
  >={Latex[width=2mm,length=2mm]},
            base/.style = {rectangle, rounded corners, draw=black,
                           minimum width=0.1cm, minimum height=1cm,
                           text centered, font=\sffamily},
           input/.style = {base, fill=blue!30, minimum width=3cm},
         control/.style = {base, fill=green!30, minimum width=3cm},
           inter/.style = {base, fill=blue!30, minimum width=1cm},
          output/.style = {base, fill=red!30, minimum width=1cm},
          action/.style = {base, fill=orange!15, minimum width=2cm, shift=3cm, rotate=90},
         process/.style = {base, minimum width=2.5cm, fill=orange!15,
                           font=\ttfamily},
}
\newsiamremark{remark}{Remark}
\newcommand{\mDG}{V_h}

\title{Medical Image Registration using optimal control of a linear hyperbolic transport equation with a DG discretization}
\author{Bastian Zapf\thanks{Expert Analytics AS, Oslo, Norway; Department of Mathematics, University of Oslo, Oslo, Norway \mbox{(bastian@xal.no)}}
\and Johannes Haubner\thanks{Department of Numerical Analysis and Scientific Computing, Simula Research Laboratory, Oslo, Norway; Institute of Mathematics and Scientific Computing, University of Graz, Austria \mbox{(johannes.haubner@uni-graz.at) }}
\and Lukas Baumg\"artner\thanks{Institut of Mathematics, Humboldt University of Berlin, 10099 Berlin, Germany \mbox{(lukas.baumgaertner@hu-berlin.de)}} \and  Stephan Schmidt\thanks{Department of Mathematics, University of Trier, 54296 Trier, Germany \mbox{(s.schmidt@uni-trier.de)}}}

\ifpdf
\hypersetup{
  pdftitle={Medical Image Registration using optimal control of a linear hyperbolic transport equation with a DG discretization},
  pdfauthor={B. Zapf, J. Haubner, L. Baumg\"artner, and S. Schmidt}
}
\fi

\begin{document}
\maketitle


\begin{abstract}
Patient specific brain mesh generation from MRI can be a time consuming task and require manual corrections, e.g., for meshing the ventricular system or defining subdomains. To address this issue, we consider an image registration approach. 
The idea is to use the registration of an input magnetic resonance image (MRI) to a respective target in order to obtain a new mesh from a template mesh. 
To obtain the transformation, we solve an optimization problem that is constrained by a linear hyperbolic transport equation. We use a higher-order discontinuous Galerkin finite element method for discretization and motivate the numerical upwind scheme and its limitations from the continuous weak space--time formulation of the transport equation. We present a numerical implementation that builds on the finite element packages FEniCS and dolfin-adjoint. 
To demonstrate the efficacy of the proposed approach, numerical results for the registration of an input to a target MRI of two distinct individuals are presented. 
Moreover, it is shown that the registration transforms a manually crafted input mesh into a new mesh for the target subject whilst preserving mesh quality.
Challenges of the algorithm are discussed. 
\end{abstract}

\tableofcontents

\section{Introduction}
The availability of high quality meshes is a crucial component in medical imaging analysis. For instance, open source high performance solver frameworks for partial differential equations (PDEs) have in recent years been applied to refine the understanding of the fundamental physiological processes governing molecular transport in the human brain \cite{kelley2022glymphatic}. Among other examples \cite{koundal2020optimal, boj2023, holter2017interstitial, vardakis2021using, valnes2020apparent} have combined PDE-constrained optimization and medical resonance images (MRI) to study molecular transport on the macroscopic scale in the human brain. One major challenge in this modelling workflow is the generation of patient-specific meshes.
To address this issue, \cite{mardal2022mathematical} recently published a recipe on how to obtain a finite element mesh of a human brain based on a MR image. 
The accompanying software \cite{svmtk} uses brain surface files generated by FreeSurfer \cite{fischl2012freesurfer} and creates volume meshes using CGAL \cite{cgal:eb-22b}.
These semi-automatically generated meshes are useful for modeling of processes in the brain parenchyma as in \cite{corti2022numerical, schreiner2022simulating, boon2022parameter}. However, in some applications like modelling of fluid transport in the cerebral aqueduct and the ventricular system \cite{hornkjol2022csf, causemann2022human} manual mesh processing is required.
This is a time-consuming task, which requires domain expertise and hence is hardly reproducible. 
In this work, we aim to simplify this workflow for future studies on new patients by utilizing manually crafted meshes to automatically generate new meshes from images.

Assume we are given two brain MRI scans from two different patients or the same patient at different times. 
Moreover, assume that we have obtained -- in a time consuming process that required manual corrections -- a  mesh for some brain regions from the first MRI. 
The question that we address is how to obtain a mesh that fits to the second MRI using the existing mesh as input.

\begin{figure}[t]
\centering
\scalebox{0.8}{
\begin{tikzpicture}[font=\sf]
\draw[-{Latex[length=0.5cm, width=0.5cm]}, line width = 0.2cm, gray!60] (0, 3.7) -- (0, 2.3);
\draw[-{Latex[length=0.5cm, width=0.5cm]}, line width = 0.2cm, gray!30] (8, 3.7) -- (8, 2.3);
\draw[-{Latex[length=0.5cm, width=0.5cm]}, line width = 0.2cm, gray!90] (3, 6) -- (5, 6);
\node[align=center] at (4, 7) { FreeSurfer \& \\ SVMTK \& \\ manual work};
\node[align=center] at (-1.5, 3) { image \\ registration \\ (Section 3)};
\node[align=center] at (9.5, 3) { mesh \\ transformation \\ (Section 4)};
\node[inner sep=0pt] at (0.,0)    {\includegraphics[width=5cm, ]{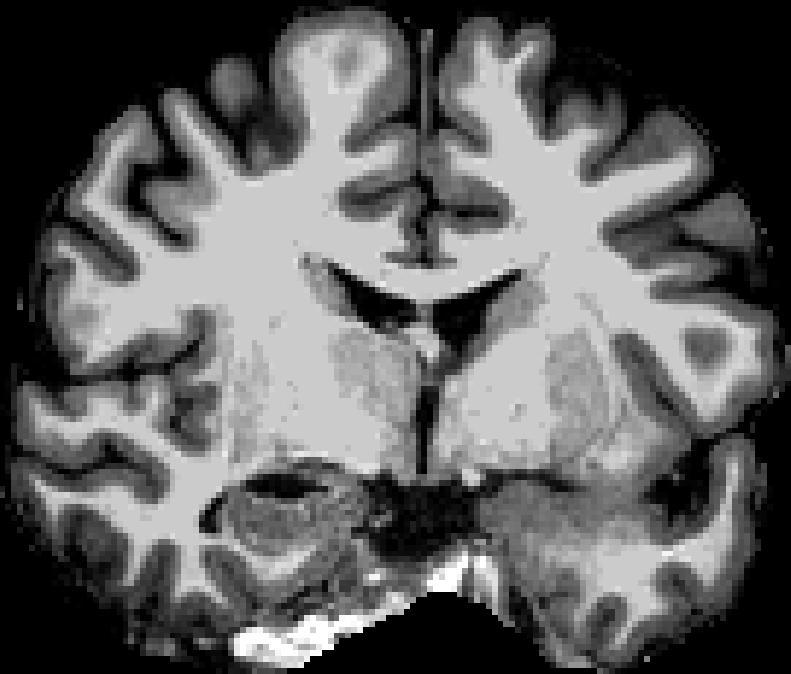}};
\node[inner sep=0pt] at (0.,6)    {\includegraphics[width=5cm, trim={0.25cm 0.25cm 0.25cm 0.25cm}, clip]{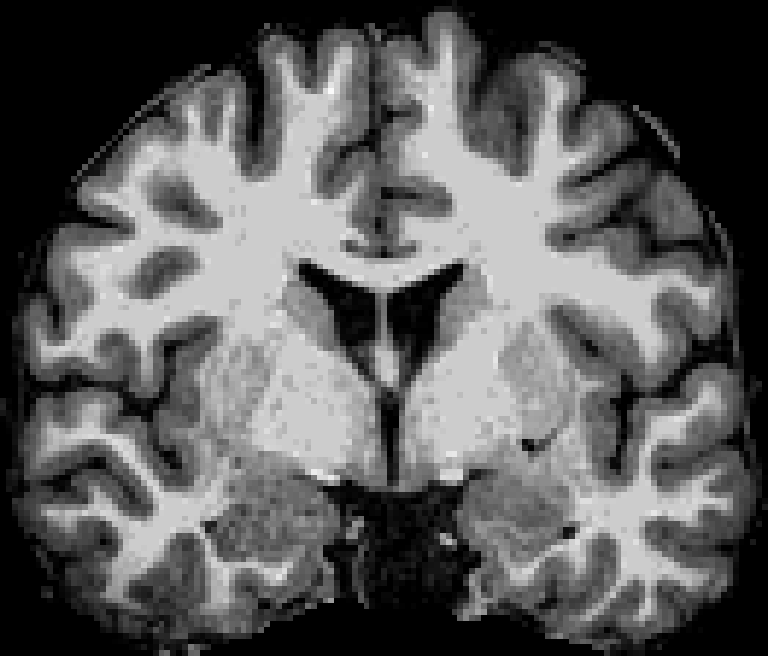}};
\node[inner sep=0pt] at (8.0, 0.1)    {\includegraphics[width=7cm, trim={15cm 0cm 15cm 0cm}, clip]{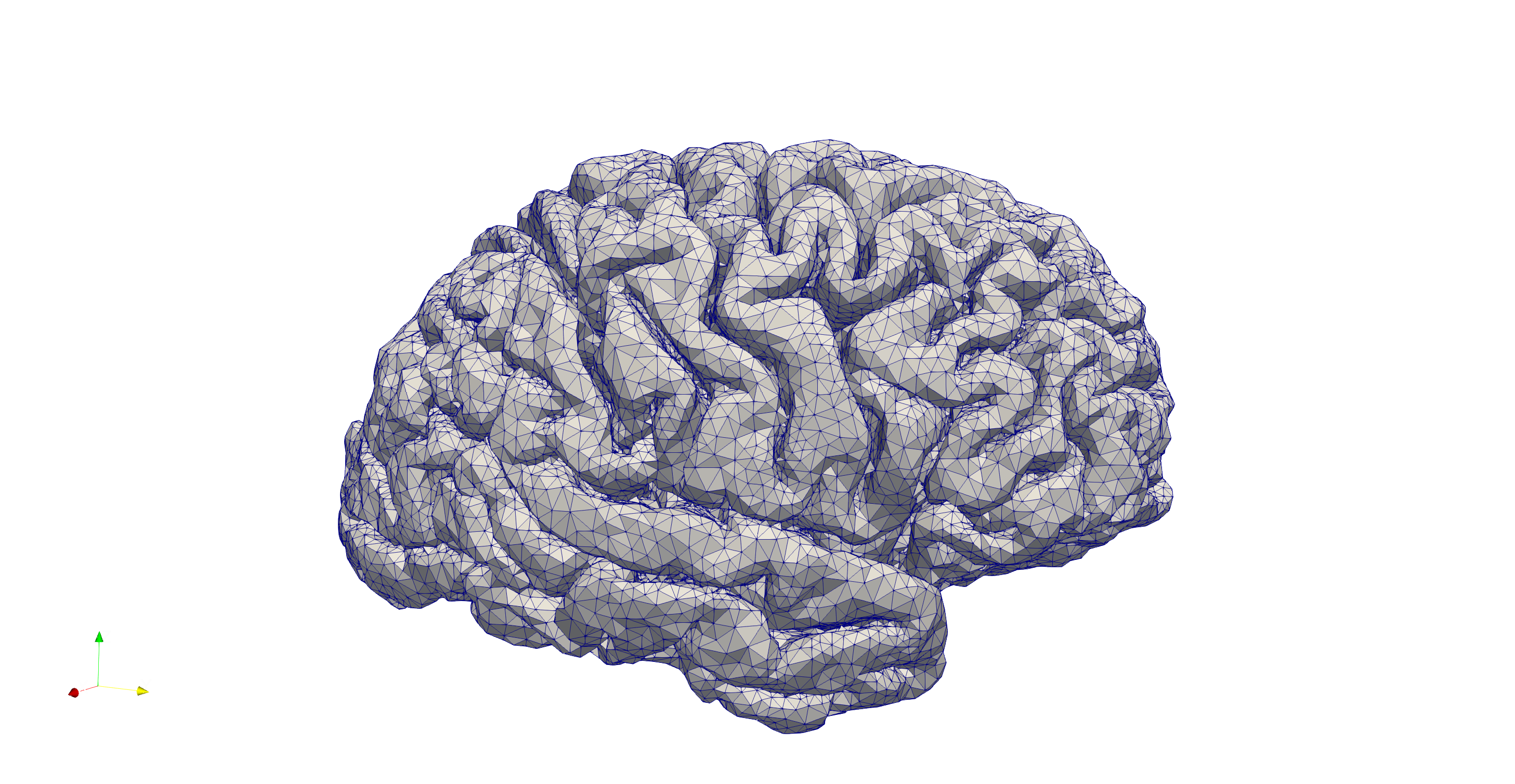}};
\node[inner sep=0pt] at (8.0,6.1)    {\includegraphics[width=7cm, trim={15cm 0cm 15cm 0cm}, clip]{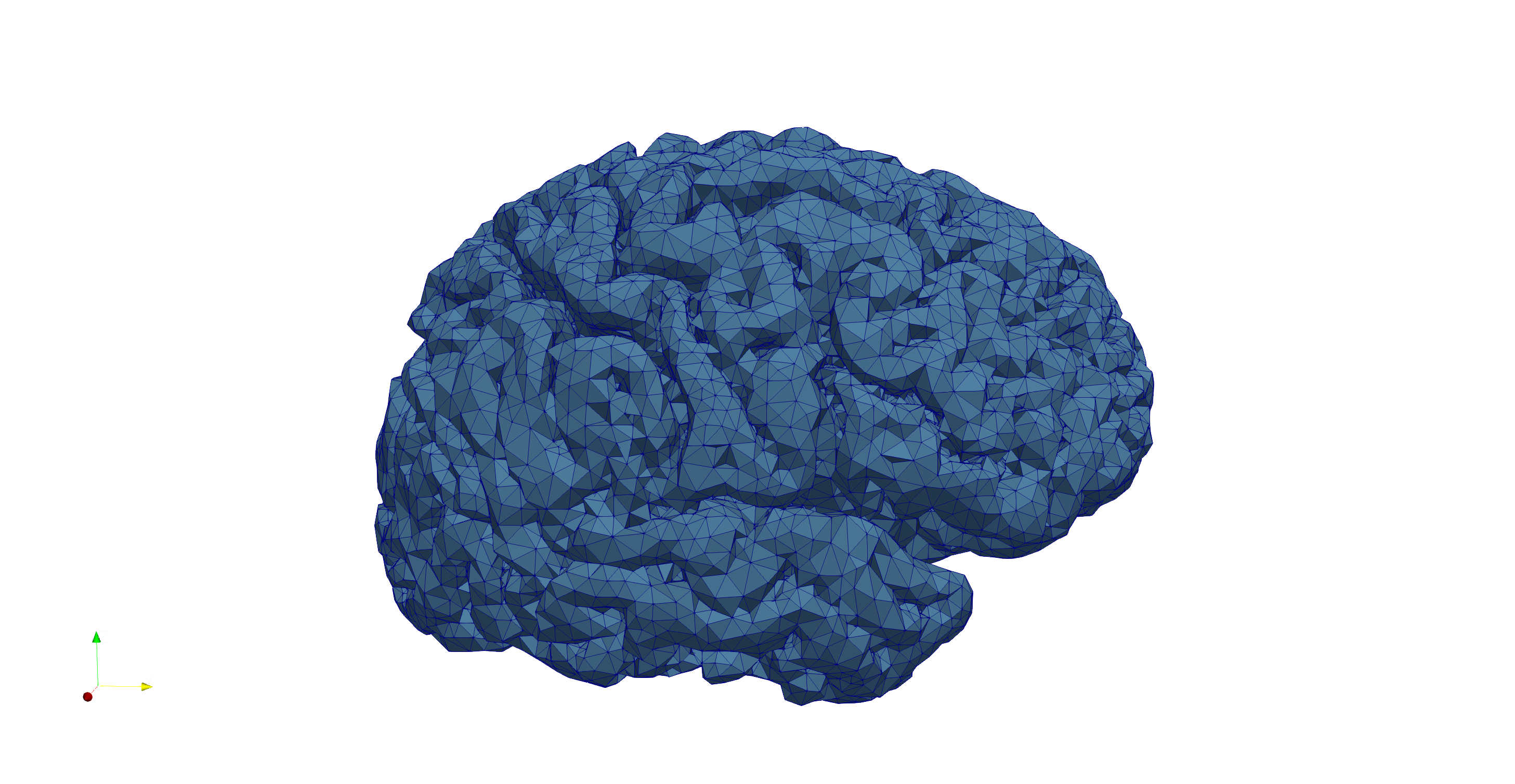}};
\node[align=left, color=white] at (-2.05, 1.9) {\textbf{Ernie}};
\node[align=left, color=white] at (-2.05, 7.9) {\textbf{Abby}};
\end{tikzpicture}
}
\caption{
Sketch of the proposed mesh deformation pipeline to generate a mesh for target \enquote{Ernie} from an existing mesh for \enquote{Abby}.
After registering the MRI from \enquote{Abby} to \enquote{Ernie}, the transformation is used to transform the mesh of \enquote{Abby} to obtain a mesh for \enquote{Ernie}.
\label{fig:idea}
}
\end{figure}
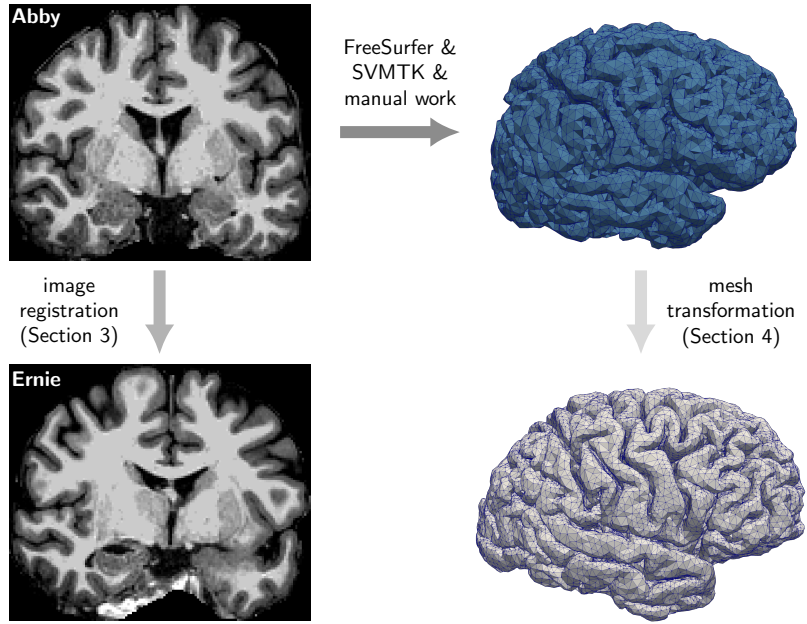

A similar question has recently been addressed in \cite{jozsa2022mri} using an affine transformation to a reference mesh. In comparison, our approach is based on nonlinear image registration and related to \cite{khan2008freesurfer}, where FreeSurfer registration was combined with large deformation diffeomorphic metric mappings \cite{beg2005computing}.
In detail, we first use an affine registration followed by a nonlinear registration via (a sequence of) deformation velocity fields.
The resulting transformation is then applied to transform the input mesh.
The proposed methodology is demonstrated using actual human MRI data. The subjects are called \enquote{Abby} and \enquote{Ernie} and are publicly available at \cite{zenodoMRIdata}.
This idea is sketched in Fig. \ref{fig:idea}. Defining the transformation via velocity fields and at the same time ensuring that the mesh quality is preserved, implies that the two brain meshes have the same topology.
Our approach to generate meshes from a new image hence implicitly requires the two brains to have the same topology. This is a common assumption in brain surface meshing \cite{acosta2012cortical, sun2019topological, segonne2007geometrically}.
We assume this to be valid in the given subject data. 
The requirement can otherwise be achieved by considering certain subregions individually.

Nonlinear registration, which is an essential part of this procedure, is a big field and not all relevant publications can be cited in this work. For a broader overview and introduction, we refer to, e.g., \cite{brown1992survey, modersitzki2009fair, mang2018pde}. Medical image registrations have been reviewed in, e.g., \cite{sotiras2013deformable, oliveira2014medical}. An experimental evaluation of different nonlinear registration algorithms can be found in \cite{klein2009evaluation}. 
In FreeSurfer \cite{fischl2012freesurfer} nonlinear intensity based registration using optical flow is a subroutine in one of the available registration algorithms \cite{postelnicu2008combined}.

We consider the image registration problem as an optimal control problem, which is governed by an optical flow, modelled by a hyperbolic transport equation as introduced in \cite{borzi2003optimal}. Analysis of the transport equation in spaces with spatial $BV$ regularity and optimal control formulations can, e.g., be found in \cite{diperna1989ordinary, ambrosio2004transport, de2008ordinary, crippa2008flow, chen2011image, jarde2018analysis, jarde2019existence}.
A numerical implementation can, e.g., be found in \cite{mang2018pde, mang2019claire, mang2015inexact, brunn2021fast}.
In recent years, image registration has been addressed via deep neural networks \cite{fu2020deep, haskins2020deep, chen2022recent}. In order to contribute to closing the gap between classical approaches and neural networks based approaches, we consider a discretization of the optimal control formulation, which can be translated into a deep neural network architecture using convolutional neural networks, CNN/MeshGraphNet and ResNet architectures. This is inspired by \cite{ruthotto2020deep}, where neural network architectures are motivated by using an explicit time-stepping for time-dependent partial differential equations. Instead of only considering the partial differential equations, we design the algorithm in such a way that the whole control-to-state map is translatable into a neural network architecture. To do so, we base our algorithm on a discretization with Discontinuous Galerkin (DG) finite elements. The sketch of the implemented image registration algorithm is presented in Fig. \ref{fig:algorithm}. Theoretical considerations show the limitations of our implementation. In this work, we implement the algorithm using classical discontinuous finite element techniques in order to do a proof-of-concept study.

We present the pipeline of our work flow. We start by preprocessing the MR images (section \ref{sec:preprocessing}) to normalize the voxel intensities and doing a first affine registration. In section \ref{sec:imageregistration}, we present our approach to image registration. We give the problem formulation, and discuss and justify our modeling choices (section \ref{subsec:justification}). For the discretization, we use a standard approach to solve the linear transport equation using DG finite elements using smoothed upwinding (section \ref{subsec:discretization}). We motivate the numerical upwind scheme from a space--time perspective (appendix \ref{sec::appendix::motivation}). Moreover, we introduce our approach to solve the resulting optimization problem (section \ref{subsec:numerics-optimization}) and present numerical results (section \ref{subsubsec:imagereg}). In section \ref{sec:meshtrafo}, we demonstrate the mesh transformation pipeline.

\section{Pre--processing}
\label{sec:preprocessing}

\begin{figure}[t]
\includegraphics[width=1.0\textwidth]{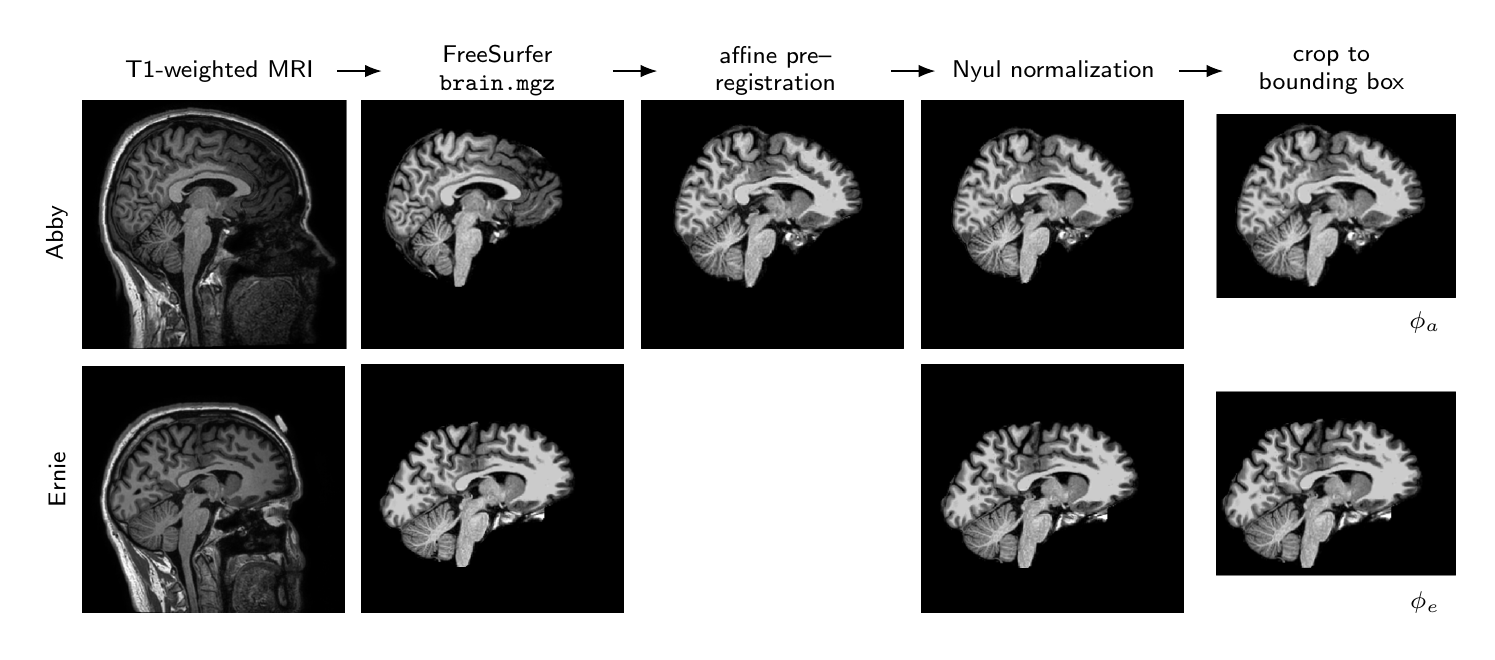}
\caption{Sketch of pre--processing pipeline. Top: Abby; Bottom: Ernie; From left to right: T1-weighted MRI, output of FreeSurfer \texttt{recon--all} (\texttt{brain.mgz}), affine pre--registration (only changes Abby image), Nyul-normalization and crop to bounding box.}
\label{fig::preprocessing}
\end{figure}

We first apply several pre--processing steps to the MR images, see Figure \ref{fig::preprocessing}.
We use the output images \texttt{brain.mgz} displaying only the brain with the cranium and other soft tissue of the head removed for both subjects, which is achieved via the \texttt{recon--all} command of FreeSurfer and already includes a first normalization step. This is beneficial since it allows us to exclude slices on the boundaries of the MRI that do not contain brain tissue, significantly reducing the computational cost of the algorithm.
Coronal slices of these images for both subjects are shown in Fig. \ref{fig:idea}.

Next, a pre--registration step using an affine transformation $\mathcal{A}$ (12 degrees of freedom) \cite{reuter2010highly} is performed.
Then the MR image intensities are normalized as proposed in \cite{nyul1999standardizing, nyul2000new} and implemented in \cite{reinhold2019evaluating}. While many normalization schemes have been proposed in the literature, e.g., \cite{friston1995spatial, sun2015histogram, christensen2003normalization}, the results in the study \cite{shah2011evaluating} suggest efficacy of the approach \cite{nyul1999standardizing, nyul2000new} and we hence resort to this normalization.

Finally, to reduce the computational cost of the nonlinear registration, we cut the images to the rectangular bounding box containing both images, but pad the image with 2 voxels containing zeros on every side (due to Dirichlet boundary condition $v=0$ for the deformation velocity).
This reduces the number of voxels from $256^3 \approx 1.7\times 10^7$ in the original images
to $n_1\times n_2 \times n_3 = 147 \times 154  \times 192 \approx 4.4\times 10^6$ whilst keeping the full brain in the image.

The cropped image intensities are then represented as a discontinuous finite element function using first order Lagrange elements defined on a Box Mesh (6 tetrahedra per voxel) with $6 \times n_1\times n_2 \times n_3$ cells and length scales $n_1\times n_2 \times n_3$.

\section{Image registration}
\label{sec:imageregistration}

We are concerned with finding a diffeomorphic mapping that transforms some input function $\phi_a$ to a target function $\phi_e$, see Fig. \ref{fig:algorithm}.
Let $\Omega \subset \mathbb R^d$, $d \in \lbrace 2, 3 \rbrace$, be a bounded Lipschitz domain, and $\phi_a,\, \phi_e,\, \phi(\cdot,T) : \Omega \to \mathbb R$ be scalar-valued. 

In this manuscript, we are concerned with the registration of MR images. Hence, in our setting, $d=3$ and $\Omega=[0, n_1]\times[0, n_2]\times[0, n_3]$, where $n_i$ denote the numbers of voxels along image axis $i\in \lbrace 1,2,3 \rbrace$.
$\phi_a$ and $\phi_e$ correspond to intensity values of the preprocessed input and target MRI, respectively.
We consider the PDE constrained optimization problem
\begin{align}
    \min_{v\in V_{ad} } \frac12 \mathcal{J}_d(\phi(\cdot,T), \phi_{e}) + \gamma \mathcal{R}(v) 
    \label{eq::reducedfunctional}
    & \\
    \begin{split}
     \text{s.t.} \ \partial_t \phi(x, t) +  v(x) \cdot \nabla \phi(x, t)  &= 0  \quad \hspace{8pt} \text{on } \Omega \times (0, T) =: Q_T, \\
     \phi(\cdot,0) &= \phi_{a} \quad \text{ in } \Omega,
     \label{eq::pde}
     \end{split}
 \end{align}
 where $V_{ad} := \lbrace v \in V ~: ~v \cdot n = 0 \text{ on } \Sigma_T := \partial \Omega \times (0,T) \rbrace$ denotes the deformation velocity field, $V$ denotes an appropriate Banach space, and $\gamma>0$ a regularization parameter.
The term $\mathcal{R}(v)$ denotes a regularization term, and the choice of regularization for the numerical results presented in manuscript are detailed in Section \ref{subsec:numerics-optimization}.
We define the data mismatch as
 \begin{align}
     \mathcal{J}_d(\phi(\cdot,T), \phi_{e}) = \int_{\Omega} f(\phi(x, T)-\phi_{e}(x)) \text{d}x
     \label{eq:data-discrepancy}
 \end{align}
 using the Huber function \cite{huber1992robust} 
 \begin{align}
    f(x) = \begin{cases} \frac{x^2}{2}  \quad &\text{if } |x| \leq \delta,  \\
    \delta \left(|x| - \frac{\delta}{2} \right) &\text{otherwise,}
    \end{cases} 
    \label{eq:huber}
\end{align}
with $\delta > 0$.

\subsection{Relation to other approaches and motivation for our choice} 
\label{subsec:justification}
A general form of the image registration problem is given as 
\begin{align}
    \min_{\tau \in \mathcal T_{ad}} \frac{1}2 \mathcal J_d( \phi_a \circ \tau, \phi_e) + \gamma \mathcal R_{\tau}(\tau).
    \label{general_opt_prob}
\end{align}
\subsubsection{Choice of the objective function}
For the choice of $f$, the square function has been considered in the image registration literature, e.g. \cite{mang2018pde}. 
To account for outliers in the image intensities, other functions are also used in brain image registration. Examples include the Huber function \cite{huber1992robust} and Tukey's biweight function \cite{tukey1960survey} as used in, e.g.,  FreeSurfer \cite{reuter2010highly}, which both limit the influence of outliers in the loss and its gradient. Similar performance of these two choices in image registration has been observed in \cite{arya2007image}. 
\begin{remark}
    Since we work with noisy images, we choose to use the Huber function.
\end{remark}

\subsubsection{Choice of the transformation} 
\label{sec::3.1.2}
There are different approaches to parametrize the transformation. Overview over techniques in image registration can also be found in \cite{zhang2018}.

\textbf{Perturbation of Identity. }
One possibility to choose the transformation is a perturbation of identity ansatz
$$\tau = \mathrm{id}+ d_\tau, $$
where $d_\tau: \Omega \to \mathbb R^d$ denotes a deformation field.
In order to ensure non-degeneracy of the transformed domain, one needs to introduce constraints. For volume preservation, \cite{haber2004numerical} worked with the constraint $\mathrm{det}(\nabla \tau) = 1$. \cite{karaccali2004estimating, zhang2018} modified this constraint to ensure $\mathrm{det}(\nabla \tau) > 0$, which, in combination with sufficient regularity, ensures diffeomorphic transformations, see \cite[Lemma 2.3]{HSU21}. A result, which works only locally around the identity and does not use the determinant constraint can already be found in \cite[Prop. 8.6]{younes2010shapes}. 

\textbf{Velocity or speed method. } Another possibility is the definition of the transformation $\tau$ via 
\begin{align} \tau_{v} := x_{v} (1), 
\label{lddmm_tau}
\end{align}
where $x_{v}$ is the solution of 
\begin{align*}
    &\partial_t x_{v} (t) = v (x_{v}(t), t), \\
    &x_{v} (0) = x.
\end{align*}
Solution theory for this setting is available for $v \in L^1((0,1), \mathcal C^{0,1}(\overline{\mathbb R}^d, \mathbb R^d))$ \cite[Sec. 4.2.1 and Example 2.1, p. 125 for the definition of $\mathcal C^{0,1}(\overline{\mathbb R}^d, \mathbb R^d)$]{delfour2011shapes}, see \cite{zolesio1979identification, zolesio1981material}. In \cite[Theorem 8.7]{younes2010shapes}, it is shown that for $v \in L^1((0,1), \mathcal C_0^1 (\overline{\Omega}, \mathbb R^d))$ the transformation $\tau_{v}$ is a diffeomorphism. In image registration, this choice for the transformation is known as LDDMM (Large deformation diffeomorphic metric mapping) algorithm \cite{trouve1998diffeomorphisms, dupuis1998variational, beg2005computing}. Theoretical guarantees for $\tau_{v}$ to be a diffeomorphism can be ensured without additional constraints and be found in, e.g., \cite[Theorem 8.7]{younes2010shapes}. The following lemma extends \cite[Theorem 8.7]{younes2010shapes} or \cite{chen2011image}, and is in the spirit of \cite[Prop. 3.4]{Jorgen} or \cite[Lemma 2.2]{HSU21}. The lemma is proven in \cite[Theorem 5.1.8]{jarde2018analysis}.

\begin{lemma}[see {\cite[Theorem 5.1.8]{jarde2018analysis}}] 
Let $\Omega \subset \mathbb R^d$, $d \geq 1$, be a bounded Lipschitz domain, $T > 0$ and $v \in L^1 ((0,T), W^{1, \infty}(\Omega)^d)$ with $v \vert_{\partial \Omega} = 0$. Moreover, for $t_0 \in (0,T)$, let $X_{t_0} $ be defined as the solution of
\begin{align*}
\begin{split}
    \partial_t X_{t_0}(x, t) = v( X_{t_0}(x,t), t) \quad & \text{on } \Omega \times (0,T), \\
    X_{t_0}(x, {t_0}) = x \quad & \text{in }\Omega.
\end{split}
\end{align*}
Then, the transformation $\tau: \Omega \to \Omega$ defined by 
$\tau(x) = \tau_{{t_0}, s}(x) := X_{t_0}(x, s)$ is bi-Lipschitz continuous for all $s \in [0,T]$ and its inverse $\tau_{{t_0}, s}^{-1}(x) = X_s(x, t_0)$ satisfies
\begin{align}
\begin{split}
    \partial_t X_{s}(x, t) = v( X_{s}(x,t), t) \quad & \text{on } \Omega \times (0,T), \\
    X_{s}(x, {s}) = x \quad & \text{in }\Omega.
\end{split}
\label{inverse}
\end{align}
Moreover, $X_{t_0}(x,s) = \eta(x,t_0,s)$, where $\eta$ satisfies the transport equation 
\begin{align*}
    \begin{split}
        \partial_t \eta (x, t, s) + (v(x, t) \cdot \nabla) \eta(x,t, s) = 0 \quad & \text{on }\Omega \times (0,T), \\
        \eta(x, s, s) = x \quad & \text{in } \Omega.
    \end{split}
\end{align*}
\label{lemma:2-1}
\end{lemma}

\begin{proof}
Due to \cite[Theorem 5.1.8 (i), (ii)]{jarde2018analysis} and \cite[Lemma 2.1]{keuthendiss}, $\tau_{t_0, t}$ is bijective and Lipschitz continuous. Due to \cite[Theorem 5.1.8 (iv)]{jarde2018analysis}, the inverse $\tau^{-1}_{t_0, t}$ is given as the solution of \eqref{inverse} and therefore also Lipschitz continuous.
The latter follows from \cite[Theorem 5.1.8 (v)]{jarde2018analysis}.
\end{proof}

\begin{remark}
Though mainly being developed independently, image registration and transformation based shape optimization are based on similar techniques. In both applications, the goal is to find transformations, which ensure that the transformed meshes or images do not degenerate. In order to ensure this, several techniques have been developed and applied in the mathematical imaging as well as the shape optimization community. 
In contrast to image registration, in shape optimization one is only interested in the shape of the domain, which is uniquely determined by its boundary. Hence, in order to have a one-to-one correspondence between shapes and transformations, in the method of mappings, see e.g. \cite{murat1975etude, brandenburg2009continuous, fischer2017frechet, keuthendiss}, the deformations are parameterized by scalar valued functions on the design boundary. In shape optimization, one is typically interested in transformations that are bi-Lipschitz, i.e., bijective and Lipschitz continuous with Lipschitz continuous inverse. If, e.g., a kink should appear after the optimization, a transformation is needed that is less regular than a $\mathcal C^1$-diffeomorphism. 
\end{remark}

\begin{remark}
    In order to avoid additional constraints in the formulation of the optimization problem, we parametrize the transformations via velocity fields that are Lipschitz continuous in space. Moreover, for the sake of simplicity and reducing computational complexity, we work with a time-independent velocity field.
\end{remark}

\subsubsection{Modeling of the optimization problem}
\label{subsec:modeling}

Different approaches to solve \eqref{general_opt_prob} with transformations defined via \eqref{lddmm_tau} have been considered. 
These can be classified in Lagrangian, Semi-\-La\-gran\-gi\-an and Eulerian methods. For a detailed discussion on that topic we refer to \cite{mang2017lagrangian}. 

Eulerian methods have the advantage that differentiating the input function is not required \cite{borzi2003optimal}. However, the memory requirements are high since, in the discretized setting, $\phi(\cdot,t_k)$ has to be stored for every time-step $t_k$ in order to compute the derivative of $\mathcal J_d$ with respect to $v$. Moreover, explicit discretization schemes need to respect a CFL condition in order to be stable.
Semi-Lagrangian methods and Lagrangian methods do not require a restrictive time step size \cite{mang2017lagrangian}. Semi-Lagrangian methods do not leverage the memory requirements compared to Eulerian methods. In addition, an interpolation step after every time-step is introduced, which can add artificial dissipation. 
For Lagrangian methods, such an interpolation step is only applied once and in a first-discretize-then-optimize framework gradients can be computed in a memory efficient way \cite{mang2017lagrangian}. However, differentiation with respect to the input function $\phi_a$ is required. 

\begin{remark}
    Since we work with noisy data and in order to avoid dealing with appropriate interpolation schemes, we choose the Eulerian framework, where it is not required to differentiate the data.
    Hence, we transform the input function $\phi_a$ directly by considering optimization problem \eqref{eq::reducedfunctional} with \eqref{eq::pde} as governing equation. 
\end{remark}

\subsection{Discretization}
\label{subsec:discretization}
Image registration based on the Eulerian framework has been approached with explicit pseudospectral RK2 schemes \cite{mang2015inexact, mang2016constrained}, explicit higher order TVD schemes \cite{benzi2011preconditioning}, implicit Lax-Friedrich schemes \cite{chen2011image} and explicit DG0 finite elements with upwinding \cite{hart2009optimal}.
Using DG finite elements seems to be a natural choice since the images are typically piecewise constant on voxels. Moreover, using a DG approach allows us to compute the inverse of the mass matrix element-wise, since the mass matrix attains a block-diagonal structure, where each block addresses the degrees of freedom of one cell. 
Solving the transport equation with non-smooth initial data leads to drawbacks, see e.g. \cite{piotrowska2019spectral} for pseudospectral methods and remark \ref{remark:jumps} for DG approaches. We use first order discontinuous Galerkin (DG) finite elements to discretize $\phi_{e}$, $\phi_{a}$ and $\phi(\cdot, t)$, since in our numerical realization we discretize trial and test functions in the same way and first order spatial derivatives appear in \eqref{eq::final_weak_form}.

\subsubsection{Upwind DG scheme}

In the literature, e.g. \cite{di2011mathematical, brezzi2004discontinuous, LeVeque, hochbruck2015efficient}, the upwind DG scheme \cite[(3.8)]{di2011mathematical} is heuristically introduced. Centered fluxes are motivated by considering the spatial weak formulation. The upwind scheme is typically motivated by adding stability to the bi-linear form by jump penalization. We give a formal motivation from a different perspective based on the space--time weak formulation in section \ref{sec::appendix::motivation}.

Let the finite element mesh on  $\Omega = \Omega_h$ be given as a collection of cells $E \in \mathcal E_h$ and interior facets $F \in \mathcal F_h$.
Let $x \in F$ and denote by $E_1$, $E_2$ the neighboring cells of the facet $F$, chosen arbitrarily but fixed. 
We then define jumps and averages analogously to \cite[Def. 1.17]{di2011mathematical} as
\begin{align*}
    &\{\!\{ v \}\!\} (x) := \frac12 (v \vert_{E_1} (x) + v \vert_{E_2} (x)), \\
    &[\![ v ]\!] (x) := v\vert_{E_1} (x) - v \vert_{E_2} (x).
\end{align*}
The outwards pointing normal vector on each element $E \in \mathcal E_h$ is denoted by $n_{E}$.
Moreover, on the exterior domain boundary, the normal $n_F = n\vert_F$ is the outwards pointing normal vector.
On interior facets, it is the normal vector pointing from $E_1(F)$ to $E_2(F)$, see \cite[Fig. 1.4]{di2011mathematical}. In the following, whenever we write $E_1$ or $E_2$, we mean $E_1(F)$ and $E_2(F)$ depending on the facet $F$ that we are considering. Moreover, let $\mDG := \mathbb P_d^k(\mathcal E_h)$, where 
\begin{align}
    \mathbb P_d^k(\mathcal E_h) := \lbrace v \in L^2(\Omega) ~: ~\forall E \in \mathcal E_h, ~v \vert_{E} \in \mathbb{P}_d^k(E) \rbrace
    \label{discrete_space_Vh}
\end{align}
and $\mathbb P_d^k(E)$ denotes the space of polynomials of $d$ variables with degree of at most $k$ on $E$ \cite[Sec. 1.2.4]{di2011mathematical}.
In addition, we define 
\begin{align*}
    &f(\phi\vert_{E_1}, \phi\vert_{E_2}, v, n_F)(x, s) := \begin{cases}
    \phi\vert_{E_1} (x, s) v (x) \cdot n_F (x) \quad &\text{if } v \cdot n_F \geq 0, \\
    \phi\vert_{E_2} (x, s) v (x) \cdot n_F (x) & \text{else},
    \end{cases}\\
    &= \phi\vert_{E_1}(x, s) \max(0, v (x) \cdot n_F (x)) + \phi\vert_{E_2}(x, s) \min(0, v (x) \cdot n_F (x)) \\
    &= \phi\vert_{E_1}(x, s) \max(0, v (x) \cdot n_F (x)) - \phi\vert_{E_2}(x, s) \max(0, v (x) \cdot (-n_F) (x)).
\end{align*}

We work with the upwind scheme
\begin{align}
\begin{split}
&  \int_\Omega \partial_t\phi (x, t) \psi(x) dx - \sum_{E \in \mathcal E_h} \int_E \mathrm{div}(v(x)) \psi(x) \phi(x, t) dx - \sum_{E \in \mathcal E_h} \int_{E} \phi(x, t) v(x) \cdot \nabla \psi(x) dx \\ & + \sum_{F \in \mathcal F_h} \int_F [\![ \psi(x) ]\!] f(\phi\vert_{E_1}(x,t), \phi\vert_{E_2}(x,t), v(x), n_F(x)) dS = 0,
\end{split}
\label{eq::final_weak_form}
\end{align}
which is equivalent to \cite[(3.7)]{di2011mathematical}. Hence, we are in the setting of \cite[Chapter 3]{di2011mathematical}.

\begin{remark}
\label{remark:jumps}
    For solutions with jumps, the literature is scarce. The upwind scheme \cite[(3.8)]{di2011mathematical} can lead to spurious oscillations for higher order DG methods \cite[Section 3.1.4.4]{di2011mathematical}. In 1d, the regions around the shocks, so-called pollution regions, can be quantified \cite{shu2016discontinuous, zhang2014error, cockburn2008error}. In order to avoid oscillations, monotonicity preserving schemes were introduced. Godunov's theorem states that linear monotonicity preserving schemes can at most be first order accurate. Further research in that area, also in the context of non-linear hyperbolic equations, is devoted to circumvent Godunov's theorem, which led to MUSCL, TVD and other schemes. The topic is discussed in more detail in \cite{lu2021oscillation}. 
\end{remark}

\begin{remark}
    In \cite[Chapter 3]{di2011mathematical}, for $\phi \in \mathcal C^0([0,T], H^1(\Omega)) \cap \mathcal C^1([0,T], L^2(\Omega))$, it is shown that the upwind DG scheme \cite[(3.8)]{di2011mathematical} combined with appropriate explicit time-stepping scheme converges towards the true solution for decreasing time step and mesh size as long as an appropriate CFL condition is fulfilled. The convergence theory presented in \cite{di2011mathematical} proves that convergence towards a true smooth solution can be ensured under a CFL condition $\Delta t \leq C \frac{h}{\| v\|_{L^\infty(\Omega)^d}}$ if one chooses a DG$0$-discretization in space combined with explicit Euler in time \cite[Theorem 3.7]{di2011mathematical} or a DG$1$-discretization in space combined with an explicit RK2 scheme in time \cite[Theorem 3.10]{di2011mathematical}. Moreover, it suggests that a more restrictive CFL condition is needed for DG$1$-discretization in space combined with explicit Euler in time \cite[Remark 3.21]{di2011mathematical}.
\end{remark}

\begin{remark}
    Due to the above considerations, we work with an explicit RK2 scheme. As discussed above, from a theoretical point of view, the input and target images $\phi_a$ and $\phi_e$ need to be either considered as discrete representations of smooth functions or preprocessed by smoothing in order to have theoretical guarantees. 
In this work, we do the former.
\end{remark}

\subsubsection{Formal motivation of upwind DG formulation as approximation of space--time variational formulation}
\label{sec::appendix::motivation}

In the literature, e.g. \cite{di2011mathematical, brezzi2004discontinuous}, the upwind DG scheme \cite[(3.8)]{di2011mathematical} is heuristically introduced by considering the spatial weak variational formulation to introduce centered fluxes and adding stability to the bilinear form by jump penalization. We give a different formal motivation by considering the space-time weak formulation. Our motivation takes advantage of standard ideas and reformulations that are, e.g., also used in \cite{LeVeque, hochbruck2015efficient} and sets them into a different perspective. 

We show that the upwind DG scheme \cite[(3.8)]{di2011mathematical} can be motivated as an approximation to the space-time weak variational form if it is assumed that $\phi \in C^1([0,T], L^2(\Omega)) \cap C([0,T], H^1(\Omega))$.
We know that for sufficiently smooth $v$ and $\phi_{a}$ 
the solution of \eqref{eq::pde} fulfills
\begin{itemize}
\item by definition the weak form 
\begin{align}
\int_0^T \int_\Omega \phi(x,s) (\partial_s \psi(x,s) + \mathrm{div}(\psi(x,s) v(x))) dx \, ds = - \int_\Omega \phi_{a}(x) \psi(x, 0) dx
\label{eq::weakform}
\end{align}
for all $\psi \in C_c^\infty([0, T) \times \Omega)$, see \cite[Definition 2.1]{jarde2019existence} or, for divergence free velocities, \cite[Definition 2]{chen2011image}. 
\item the relation 
\begin{align}
\phi(x,t) = \phi_{a}(X(x,t)),
\label{eq::solrel}
\end{align}
where $X$ is the solution of
\begin{align}
\begin{split}
\partial_t X - \nabla X^\top v = 0 \quad &\text{on }Q_T, \\
X(0) = x \quad & \text{in } \Omega,
\end{split}
\label{eq::pdeX}
\end{align}
see \cite[Theorem 4]{chen2011image} or, under mild regularity assumptions, \cite[Theorem 5.1.15]{jarde2018analysis}.
\end{itemize}

For the sake of simplicity and in correspondence with our choice in section \ref{sec::3.1.2}, we consider velocity fields $v \in W^{1, \infty}(\Omega)^d$ such that $v\vert_{\Omega} = 0$ on the outer boundary of $\Omega$.
In order to derive the upwind DG scheme \cite[(3.8)]{di2011mathematical}, we are going to use \eqref{eq::weakform} and \eqref{eq::solrel}. 

Integrating the weak form \eqref{eq::weakform} by parts in space yields
\begin{align*}
\begin{split}
    &\int_0^T \int_\Omega \phi(x,s) \partial_s \psi(x,s)\, dx\, ds  - \int_0^T \int_\Omega \psi(x, s) v(x) \cdot \nabla \phi(x, s)\, dx\, ds = - \int_\Omega \phi_{a}(x) \psi(x, 0) dx.
\end{split}
\end{align*}
Collecting the integrals over the elements as integrals over facets yields
\begin{align}
    \begin{split}
    \int_0^T \int_\Omega \phi(x,s) \partial_s \psi(x,s)\, dx\, ds  - \int_0^T \sum_{E \in \mathcal E_h} \int_E \psi(x, s) v(x) \cdot \nabla \phi(x, s)&\, dx\, ds \\
    & = - \int_\Omega \phi_{a}(x) \psi(x, 0) dx,
\end{split}
\label{eq::ds3}
\end{align}

Since, in \eqref{eq::ds3}, there is no derivative on $\psi$ left, we can leverage the requirements on $\psi$ by considering \eqref{eq::ds3} for a sequence of smooth functions $\psi_n \in C_c^\infty([0, T) \times \Omega)$ that we pass to the limit. Let $\psi \in C^1([0,T],V_h)$. Choosing $\psi_n$ such that $\lim_{n \to \infty} \| \psi_n - \psi 1_{E}\|_{L^p(\Omega)} = 0$ for some $p \in [1, \infty)$ and $1_E$ denoting the characteristic function of $E$, and summing up over the elements yields
\begin{align}
    \begin{split}
       \int_0^T \int_\Omega \phi(x,s) \partial_s \psi(x,s)\, dx\, ds  &- \int_0^T \sum_{E \in \mathcal E_h} \int_E \psi(x, s) v(x) \cdot \nabla \phi(x, s)\, dx\, ds = - \int_\Omega \phi_{a}(x) \psi(x, 0) dx 
       \label{eq::ds4}
    \end{split}
\end{align}
for all $\psi \in C^1([0,T), \mDG)$.

Integrating \eqref{eq::ds4} by parts yields 
\begin{align*}
\begin{split}
    &\int_0^T \int_\Omega \phi(x,s) \partial_s \psi(x,s)\, dx\, ds  + \int_0^T \sum_{E \in \mathcal E_h} \int_E \mathrm{div}(\psi(x, s) v(x)) \phi(x, s)\, dx\, ds  \\
    & - \int_0^T \sum_{F \in \mathcal F_h} \int_F [\![ \psi(x, s) \phi(x, s) ]\!] v \cdot n_F  \, dS\, ds = - \int_\Omega \phi_{a}(x) \psi(x, 0) dx ,
    \end{split}
\end{align*}
which is equivalent to
\begin{align}
\begin{split}
    &\int_0^T \int_\Omega \phi(x,s) \partial_s \psi(x,s)\, dx\, ds  + \int_0^T \sum_{E \in \mathcal E_h} \int_E \mathrm{div}(v(x)) \psi(x, s) \phi(x, s)\, dx\, ds  \\
    &+ \int_0^T \sum_{E \in \mathcal E_h} \int_E  \phi(x, s) v(x) \cdot \nabla \psi(x, t)\, dx\, ds   - \int_0^T \sum_{F \in \mathcal F_h} \int_F [\![ \psi(x, t) \phi(x, s) ]\!] v \cdot n_F  \, dS\, ds \\
    & = - \int_\Omega \phi_{a}(x) \psi(x, 0) dx.
    \end{split}
    \label{eq::ds_25}
\end{align}
Let $t \in [0,T)$, $\Delta t > 0$ be such that $t + \Delta t \in (0,T]$. We do an integration by parts in time and consider a sequence of test function $\psi^{t, \Delta t}_n(x, s) = \psi(x) \chi^{t, \Delta t}_n(s)$ with a time-independent $\psi$ and $\chi^{t, \Delta t}_n \to \chi_{t, \Delta t}$ for $n \to \infty$ in $L^2((0,T), V_h)$ with $\chi_{t, \Delta t}(s) = \begin{cases}
1 \quad & \text{for } s \in [t, t + \Delta t), \\
0 \quad & \text{else}. \\
\end{cases}$\\
Hence,
\eqref{eq::ds_25} tested with $\psi_{t, \Delta t}$, is given by
\begin{align*}
\begin{split}
 & - \int_t^{t + \Delta t} \int_\Omega \partial_t \phi(x, s) \psi(x)  \, dx ds + \int_{t}^{t + \Delta t} \sum_{E \in \mathcal E_h} \int_E \mathrm{div}(v(x)) \psi(x) \phi(x, s)\, dx\, ds \\ 
 &+ \int_{t}^{t + \Delta t} \sum_{E \in \mathcal E_h} \int_E  \phi(x, s) v(x) \cdot \nabla \psi(x)\, dx\, ds   - \int_{t}^{t + \Delta t} \sum_{F \in \mathcal F_h} \int_F [\![ \psi(x) \phi(x, s) ]\!] v \cdot n_F  \, dS\, ds = 0,
 \end{split}
\end{align*}
if $t> 0$, and 
\begin{align*}
\begin{split}
 & - \int_0^{\Delta t} \int_\Omega \partial_t \phi(x,s) \psi(x) \, dx - \int_\Omega \phi(x, 0) \psi(x) \, dx + \int_{0}^{\Delta t} \sum_{E \in \mathcal E_h} \int_E \mathrm{div}(v(x)) \psi(x) \phi(x, s)\, dx\, ds \\ 
 &+ \int_{0}^{\Delta t} \sum_{E \in \mathcal E_h} \int_E  \phi(x, s) v(x) \cdot \nabla \psi(x)\, dx\, ds   - \int_{t}^{t + \Delta t} \sum_{F \in \mathcal F_h} \int_F [\![ \psi(x) \phi(x, s) ]\!] v \cdot n_F  \, dS\, ds \\ & = - \int_\Omega \phi_{a}(x) \psi(x) dx,
 \end{split}
\end{align*}
if $ t = 0$. 
Integration by parts in time yields
\begin{align*}
\begin{split}
 & - \int_\Omega (\phi(x, t + \Delta t) - \phi(x, t)) \psi(x)  \, dx ds + \int_{t}^{t + \Delta t} \sum_{E \in \mathcal E_h} \int_E \mathrm{div}(v(x)) \psi(x) \phi(x, s)\, dx\, ds \\ 
 &+ \int_{t}^{t + \Delta t} \sum_{E \in \mathcal E_h} \int_E  \phi(x, s) v(x) \cdot \nabla \psi(x)\, dx\, ds   - \int_{t}^{t + \Delta t} \sum_{F \in \mathcal F_h} \int_F [\![ \psi(x) \phi(x, s) ]\!] v \cdot n_F  \, dS\, ds = 0,
 \end{split}
\end{align*}
if $t> 0$, and 
\begin{align*}
\begin{split}
 & -\int_\Omega \phi(x,\Delta t) \psi(x) \, dx  + \int_{0}^{\Delta t} \sum_{E \in \mathcal E_h} \int_E \mathrm{div}(v(x)) \psi(x) \phi(x, s)\, dx\, ds \\ 
 &+ \int_{0}^{\Delta t} \sum_{E \in \mathcal E_h} \int_E  \phi(x, s) v(x) \cdot \nabla \psi(x)\, dx\, ds   - \int_{t}^{t + \Delta t} \sum_{F \in \mathcal F_h} \int_F [\![ \psi(x) \phi(x, s) ]\!] v \cdot n_F  \, dS\, ds \\ & = - \int_\Omega \phi_{a}(x) \psi(x) dx,
 \end{split}
\end{align*}
if $ t = 0$.

At this point, we approximate $\phi_a$ and $\phi$ by a piecewise smooth function $\bar \phi_a \in V_h$ and $\bar \phi (x,t) = \bar \phi_a \circ X(x,t)$. Doing this approximation, we consider
\begin{align}
\begin{split}
 & - \int_\Omega (\bar \phi(x, t + \Delta t) - \bar \phi(x, t)) \psi(x) \, dx + \int_{t}^{t + \Delta t} \sum_{E \in \mathcal E_h} \int_E \mathrm{div}(v(x)) \psi(x) \bar \phi(x, s)\, dx\, ds \\ 
 &+ \int_{t}^{t + \Delta t} \sum_{E \in \mathcal E_h} \int_E  \bar \phi(x, s) v(x) \cdot \nabla \psi(x)\, dx\, ds   - \int_{t}^{t + \Delta t} \sum_{F \in \mathcal F_h} \int_F [\![ \psi(x) \bar \phi(x, s) ]\!] v \cdot n_F  \, dS\, ds = 0,
 \end{split}
 \label{eq::integral1}
\end{align}
if $t> 0$, and 
\begin{align*}
\begin{split}
 & - \int_\Omega (\bar \phi(x, t + \Delta t)) \psi(x) \, dx + \int_{t}^{t + \Delta t} \sum_{E \in \mathcal E_h} \int_E \mathrm{div}(v(x)) \psi(x) \bar \phi(x, s)\, dx\, ds \\ 
 &+ \int_{t}^{t + \Delta t} \sum_{E \in \mathcal E_h} \int_E  \bar \phi(x, s) v(x) \cdot \nabla \psi(x)\, dx\, ds   - \int_{t}^{t + \Delta t} \sum_{F \in \mathcal F_h} \int_F [\![ \psi(x) \bar \phi(x, s) ]\!] v \cdot n_F  \, dS\, ds \\ & = - \int_\Omega \bar \phi_{a}(x) \psi(x) dx,
 \end{split}
\end{align*}
if $ t = 0$.

For the latter, considering $\Delta t \to 0$ yields the initial condition
\begin{align}
\int_\Omega \bar \phi(x, 0) \psi(x) dx = \int_\Omega \bar \phi_{a} (x) \psi(x) dx.
\end{align}
For the consideration of \eqref{eq::integral1}, we define 
\begin{align*}
    &f(\bar \phi\vert_{E_1}, \bar \phi\vert_{E_2}, v, n_F)(x, s) := \begin{cases}
    \bar \phi\vert_{E_1} (x, s) v (x) \cdot n_F (x) \quad &\text{if } v \cdot n_F \geq 0, \\
    \bar \phi\vert_{E_2} (x, s) v (x) \cdot n_F (x) & \text{else},
    \end{cases}\\
    &= \bar \phi\vert_{E_1}(x, s) \max(0, v (x) \cdot n_F (x)) + \bar \phi\vert_{E_2}(x, s) \min(0, v (x) \cdot n_F (x)) \\
    &= \bar \phi\vert_{E_1}(x, s) \max(0, v (x) \cdot n_F (x)) - \bar \phi\vert_{E_2}(x, s) \max(0, v (x) \cdot (-n_F) (x)).
\end{align*}

Since we consider the space-time weak formulation (instead of the weak formulation in space), upwind-type fluxes appear as natural choice in the DG scheme (instead of centered fluxes). Let $X$ be defined as the solution of \eqref{eq::pdeX}. Moreover, \eqref{eq::solrel} holds, i.e. $\phi$ is transported along the characteristics of $X$ such that $\bar \phi(x,s) = \bar \phi(X(x,t-s), t)$. If $v(x) \cdot n_F > 0$ ($v(x) \cdot n_F <0$) and $s-t>0$ is sufficiently small, then $X(x, t-s) \in E_1$ ($X(x, t-s) \in E_2$). 
Therefore, for $x \in \bigcup_{F \in \mathcal F_h} F$, 
$$\lim_{s \to t^+} \bar \phi (X(x, t-s), t) v \cdot n_F = f(\bar \phi\vert_{E_1}(x,t), \bar \phi\vert_{E_2}(x,t), v(x), n_F(x)).$$ 
Hence, the point-wise limit is given as
\begin{align*}
    \lim_{s \to t^+} [\![ \psi(x) \bar \phi(x,s) ]\!] v \cdot n_F = [\![\psi (x)]\!] f(\bar \phi\vert_{E_1}(x,t), \bar \phi\vert_{E_2}(x,t), v(x), n_F(x))
\end{align*}
since 
\begin{itemize}
    \item either there is a jump of $\bar \phi$ at $(x,t)$, i.e. $\bar \phi(x,t)\vert_{E_1} \neq \bar \phi(x,t)\vert_{E_2}$: The case $v \cdot n_F = 0$ is trivial. For $v \cdot n_F > 0$ ($v \cdot n_F <0$) it holds that $X(x,t-s) \in E_1$ ($X(x,t-s) \in E_2$) and, therefore, $\bar \phi(x, s)\vert_{E_1} = \bar \phi(x, s)\vert_{E_2}$ for $s-t> 0$ sufficiently small. 
    \item or there is no jump of $\bar \phi$ at $(x,t)$, i.e., there exists a $\epsilon = \epsilon(x) > 0$ such that $\phi(x,s)\vert_{E_1} = \phi(x,s) \vert_{E_2}$ for $s \in [t, t+ \epsilon)$.
\end{itemize}
From the dominated convergence theorem it follows that 
\begin{align*}
\begin{split}
&\lim_{\Delta t \to 0^+} \frac{1}{\Delta t} \int_{t}^{t+\Delta t} \int_F [\![ \psi (x) \bar \phi(x, s) ]\!] v \cdot n_F dS ds \\  &= \int_F [\![\psi (x)]\!] f(\bar \phi\vert_{E_1}(x,t), \bar \phi\vert_{E_2}(x,t), v(x), n_F(x)) dS
\end{split}
\end{align*}
for all $F \in \mathcal F_h$.
\label{lemma::1}

Moreover, it holds that
\begin{align}
\lim_{\Delta t \to 0^+} \frac{1}{\Delta t}\int_{\Omega} ( \bar \phi(x,t + \Delta t) - \bar \phi(x, t)) \psi(x) dx = \int_\Omega \bar \phi_t(x, t) \psi(x) dx,
\label{eq::lemma42::1}
\end{align}
and
\begin{align}
\lim_{\Delta t \to 0^+} \frac{1}{\Delta t} \int_{t}^{t + \Delta t} \int_E \mathrm{div}(v(x)) \psi(x) \bar \phi(x,s) dx \, ds = \int_E \mathrm{div}(v(x)) \psi(x) \bar \phi(x,t) dx,
\label{eq::lemma42::2}
\end{align}
\begin{align}
\lim_{\Delta t \to 0^+} \frac{1}{\Delta t} \int_{t}^{t + \Delta t} \int_E \bar \phi(x,s) v(x) \cdot \nabla \psi(x) dx ds = \int_E \bar \phi(x,t) v(x) \cdot \nabla \psi(x) dx,
\label{eq::lemma42::3}
\end{align}
for all $E \in \mathcal E_h$.
\label{lemma::2}

This can be shown as follows: 
\eqref{eq::lemma42::1} follows from the dominated convergence theorem and the regularity of $\bar \phi$. To prove \eqref{eq::lemma42::2} we consider 
\begin{align*}
& | \frac{1}{\Delta t} \int_t^{t + \Delta t} \int_E \mathrm{div}(v(x)) \psi (x) \bar \phi(x,s) dx ds - \int_E \mathrm{div}(v(x)) \psi(x) \bar \phi(x,t) dx | \\
& = \frac{1}{\Delta t} | \int_t^{t + \Delta t} \int_E \mathrm{div}(v(x)) \psi(x) (\bar \phi(x,s) - \bar \phi(x,t)) dx ds| \\
& \leq \| \mathrm{div}(v(x)) \psi(x)\|_{L^\infty(E)} \max_{s \in [t, t + \Delta t]} \| \bar \phi(x,s) - \bar \phi(x,t) \|_{L^1(E)}
\end{align*}
and use continuity of $\bar \phi$ to show that the right hand side converges to zero for $\Delta t \to 0$. \eqref{eq::lemma42::3} follows analogously.

Hence, dividing \eqref{eq::integral1} by $\Delta t$ and considering the limit $\Delta t \to 0^+$ yields:
\begin{align*}
\begin{split}
&  \int_\Omega \partial_t \bar \phi (x, t) \psi(x) dx - \sum_{E \in \mathcal E_h} \int_E \mathrm{div}(v(x)) \psi(x) \bar \phi(x, t) dx - \sum_{E \in \mathcal E_h} \int_{E} \bar\phi(x, t) v(x) \cdot \nabla \psi(x) dx \\ & + \sum_{F \in \mathcal F_h} \int_F [\![ \psi(x) ]\!] f(\bar \phi\vert_{E_1}(x,t), \bar \phi\vert_{E_2}(x,t), v(x), n_F(x)) dS = 0.
\end{split}
\end{align*}

\subsubsection{Smoothed upwinding}
We want to use gradient-based optimization techniques to determine the unknown velocity $v$ that deforms $\phi_a$ to a target $\phi_e$. 
The discrepancy measure $\mathcal{J}_d$ \eqref{eq::reducedfunctional} needs to be differentiable with respect to $v$. 
This requires differentiability of the solution to \eqref{eq::final_weak_form} with respect to $v$.
However, this is in general not the case.
Due to the $\max$ operator, the upwind flux function in the \eqref{eq::final_weak_form} is not differentiable w.r.t. $v$ in all points $x$, for which $v(x)\cdot n(x) = 0$. 
Therefore, we introduce a smoothed maximum operator defined as 
\begin{align*}
    \operatorname{max}_\epsilon(0,x) := \sigma_\epsilon(x)x
\end{align*}
with
\begin{align*}
    \sigma_\epsilon(x) = \frac{1}{1+e^{-\frac{1}\epsilon x}}
\end{align*}
and propose a smoothed version of the upwind scheme by using
\begin{align}
    f_\epsilon(\phi\vert_{E_1}, \phi\vert_{E_2}, v, n_F) &:= \phi\vert_{E_1} \operatorname{max}_\epsilon(0, v \cdot n_F) - \phi\vert_{E_2} \operatorname{max}_\epsilon(0, v \cdot (-n_F)).\label{eq::smoothed-flux}
\end{align}
This flux function is consistent since, if $\phi\vert_{E_1} = \phi\vert_{E_2} = \bar \phi $,
\begin{align*}
    f_\epsilon(\bar{\phi}, \bar{\phi}, v, n_F) &= \bar{\phi} \operatorname{max}_\epsilon(0, v \cdot n_F) - \bar{\phi} \operatorname{max}_\epsilon(0, v \cdot (-n_F)) \\
     &= \Bar{\phi} [ \sigma_\epsilon(v\cdot n_F)v\cdot n_F - \sigma_\epsilon(v\cdot (-n_F) ) v\cdot (-n_F) ] \\
     &= \Bar{\phi} [\sigma_\epsilon(v\cdot n_F) + \sigma_\epsilon(-v\cdot n_F)] v\cdot n_F \\
     &= \Bar{\phi} v\cdot n_F,
\end{align*}
where we used that the sigmoid function satisfies $\sigma_\epsilon(x) + \sigma_\epsilon(-x) = 1$.

\subsection{Optimization} 
\label{subsec:numerics-optimization}
For solving the optimization problem \eqref{eq::reducedfunctional}--\eqref{eq::pde} and its (Semi-)
Lagrangian variants, different approaches have been considered and discussed in the literature. These include gradient descent methods, e.g. \cite{chen2011image}, and inexact Newton and Gauss-Newton schemes, e.g. \cite{mang2015inexact, benzi2011preconditioning}. In its reduced form the image registration problem \eqref{eq::reducedfunctional}--\eqref{eq::pde} is given by
\begin{align}
\min_{v \in V_{ad}} j(v) + \gamma \mathcal R(v),
\label{optprob::reduced}
\end{align}
where $j(v) := \frac12 \mathcal J_d (\phi(\cdot, T), \phi_e)$ and $\phi(\cdot, T)$ solves \eqref{eq::pde}. Since \eqref{optprob::reduced} is an unconstrained optimization problem, solving it can be done by considering the first-order optimality conditions and solving the corresponding nonlinear system of equations iteratively with preconditioning, see e.g. \cite{mang2015inexact, benzi2011preconditioning}. Another approach, which is more related to current machine learning approaches and which we follow in this work, is solving \eqref{optprob::reduced} via general purpose optimization solvers.
In the following, we discuss the choice of $\mathcal R(v)$ and $V_{ad}$ and its discretization, and apply a state-of-the-art optimization solver to solve \eqref{optprob::reduced}. 

Due to the considerations in Lemma \ref{lemma:2-1}, the choice $V_{ad} = W^{1, \infty}(\Omega)$ is suitable.
Optimization in $W^{1, \infty}(\Omega)$ is challenging and subject to current research \cite{deckelnick2022novel}.
We consider an alternative approach that works in a Hilbert space setting, i.e., optimization for $v \in H^2(\Omega)$. This is justified from a discretized perspective, since discretization with CG1 finite elements combined with the discretized $H^k$-norm, $k \geq 1$, ensures the regularity requirements of Lemma \ref{lemma:2-1} if the mesh is fixed. 
In order to avoid a discretization of the velocity $v$ with $H^2$-conforming finite elements, we introduce an auxiliary variable $\tilde v \in L^2(\Omega)$ from which $v$ is computed via solving the elliptic problem
\begin{align}
\begin{split}
       \alpha v - \beta \Delta v & = \tilde v \quad \text{in } \Omega, \\
    v &= 0 \quad \text{on } \partial \Omega.
    \label{math:velocity-pre--processing}
\end{split}
\end{align}
Here, $\alpha, \beta > 0$ are hyperparameters and their influence are discussed in the numerical results in Section \ref{subsubsec:imagereg}.
We hence optimize $\mathcal{J}_d$ for $\tilde v \in L^2(\Omega)$. General-purpose optimization algorithms, such as L-BFGS-B \cite{liu1989limited} implemented in SciPy~\cite{virtanen2020scipy}, work with the Euclidean inner product $\ell^2$.
Directly using the nodal values $ \tilde{\mathbf v}_i$ of $\tilde v$ as $\ell^2$ control variable can lead to slow convergence \cite{schwedes2017mesh, kelley1987quasi}. As in \cite[Sec. 7.2]{haubner2023novel}, we introduce a way to respect the function space perspective.
We use bold-face symbols to denote the degrees of freedom of the corresponding finite element function and introduce the control variable $\hat{\mathbf v} \in \mathbb{R}^n$ where $n$ is the number of degrees of freedom of the finite element discretization of $\tilde v$. The former are related via the linear transformation
\begin{align}
    \tilde{\mathbf v} = \mathcal{C}^{-1} \hat{\mathbf v}, \label{math:L2l2transform}
\end{align}
where $\mathcal{C}^\top \mathcal{C} = \mathcal{M}$ is the Cholesky factorization of the inner product matrix $\mathcal{M}$. In our case, $\mathcal{M}$ is the lumped mass matrix and hence the linear transformation \eqref{math:L2l2transform} is a scaling and promotes mesh independence. We note that directly working with control variables $v \in H^k(\Omega)$ and applying the linear transformation requires computing the action of the inverse of the Cholesky factors of the stiffness matrix. This is currently not implemented in parallel in our software framework, and therefore we use the approach presented above. A graphical summary of the approach is given in Fig. \ref{fig:algorithm}.

We work with the regularization (see \eqref{eq::reducedfunctional}) defined by
\begin{align*}
    \mathcal{R}(v)=\mathcal{R}(v(\tilde v)) = \int_{\Omega} \|\tilde v(x)\|^2 \,dx. 
\end{align*}
Alternatively, one can also directly work with $\hat{\mathbf v}\cdot \hat{\mathbf v}$.

\begin{sidewaysfigure}[p]
    \centering
    \includegraphics[width=1.0\textwidth]{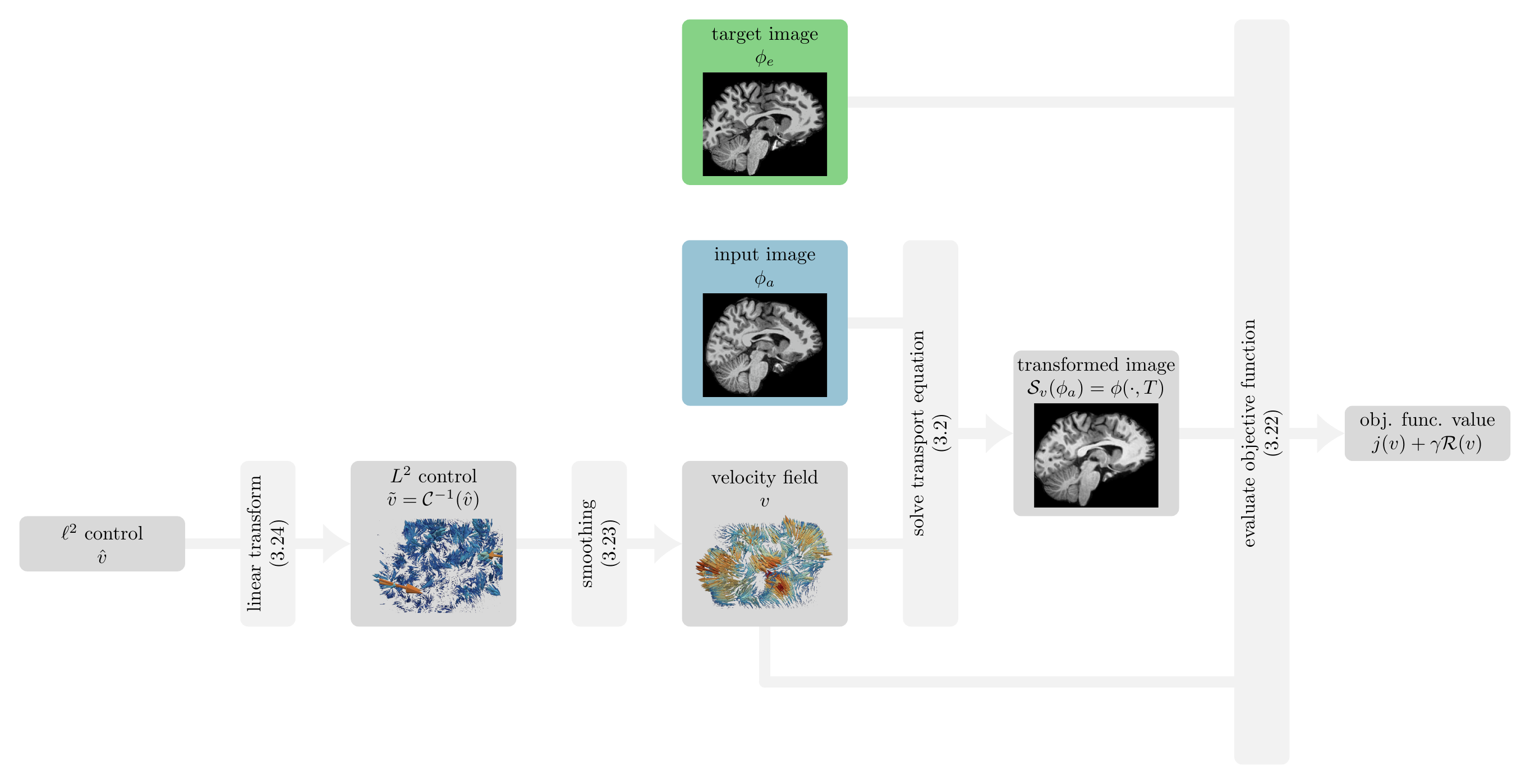}
    \caption{Implementation of the velocity-based image registration.
    }
    \label{fig:algorithm}
\end{sidewaysfigure}

\subsection{Numerical results}\label{subsubsec:imagereg}

We implement \eqref{eq::final_weak_form} with the smoothed flux term \eqref{eq::smoothed-flux} in FEniCS~\cite{alnaes2015fenics} using the explicit midpoint method, a second order Runge--Kutta scheme in time. For the results reported in this manuscript, we use fixed $N=100$ time steps and a time horizon $T=1$.
We discretize $v$ using continuous linear Lagrange elements.
The initial condition $\phi_a$ and the state $\phi$ are discretized by a discontinuous Galerkin method using first order polynomials. 

We introduce the discrete solution operator $\mathcal{S}_v$, mapping from a given initial condition $\phi_0$ to the final state, as
\begin{align*}
    {\mathcal S}_v( \phi_0) := \phi(\cdot,T),
\end{align*}
where $\phi$ solves the (DG1 discretized) PDE
\begin{align*}
     \partial_t \phi(x, t) +  v(x) \cdot \nabla \phi(x, t)  &= 0  \quad \hspace{8pt} \text{on } Q_T, \\
     \phi(\cdot,0) &= \phi_{0} \quad \text{ in } \Omega,
\end{align*}
as introduced in \eqref{eq::pde}.
We use the pre--processed MRI $\phi_a$ from \enquote{Abby} as initial condition $\phi_0$ and the starting guess $\hat{\mathbf v}=0$ to solve \eqref{eq::pde} numerically using the proposed discretization implemented in FEniCS~\cite{alnaes2015fenics} as described in Figure \ref{fig:algorithm}.
Gradients of the cost functional with respect to the control $\hat v$ are obtained using the dolfin-adjoint software~\cite{farrell2013automated}. 

To allow for more complex deformations, we take a multi-step approach, cf. remark \ref{sec:remark-time-dependency} and Fig. \ref{fig:multistep}. 
In detail, we start with the affine registered image from \enquote{Abby} and optimize for the first velocity field $v=v_1$ with a fixed number of iterations of L-BFGS-B. 
Starting from the resulting deformed image, we optimize for a new deformation velocity field $v_2$. 
Here, we repeat this scheme four times.
We shortly write $\mathcal S_i$ for velocity fields $v = v_i$, where $v_i$ denotes the optimal velocity found via solving \eqref{eq::reducedfunctional}--\eqref{eq::pde} with initial condition $\phi_{i-1}$. Hence, starting with the initial condition
\begin{math}
    \phi_0 = \phi_a,
\end{math}
we obtain
\begin{align*}
    & \phi_1 = \mathcal S_1( \phi_0), \quad\phi_2 = \mathcal S_2(\phi_1), \quad \phi_3 =  \mathcal S_3(\phi_2), \quad \phi_4 = \mathcal S_4(\phi_3). 
\end{align*}
This approach is visualized in Fig. \ref{fig:multistep}.
Recall that $\phi_a$ denotes the preprocessed template image from \enquote{Abby}.
\begin{figure}
\centering
\includegraphics[width=1\textwidth]{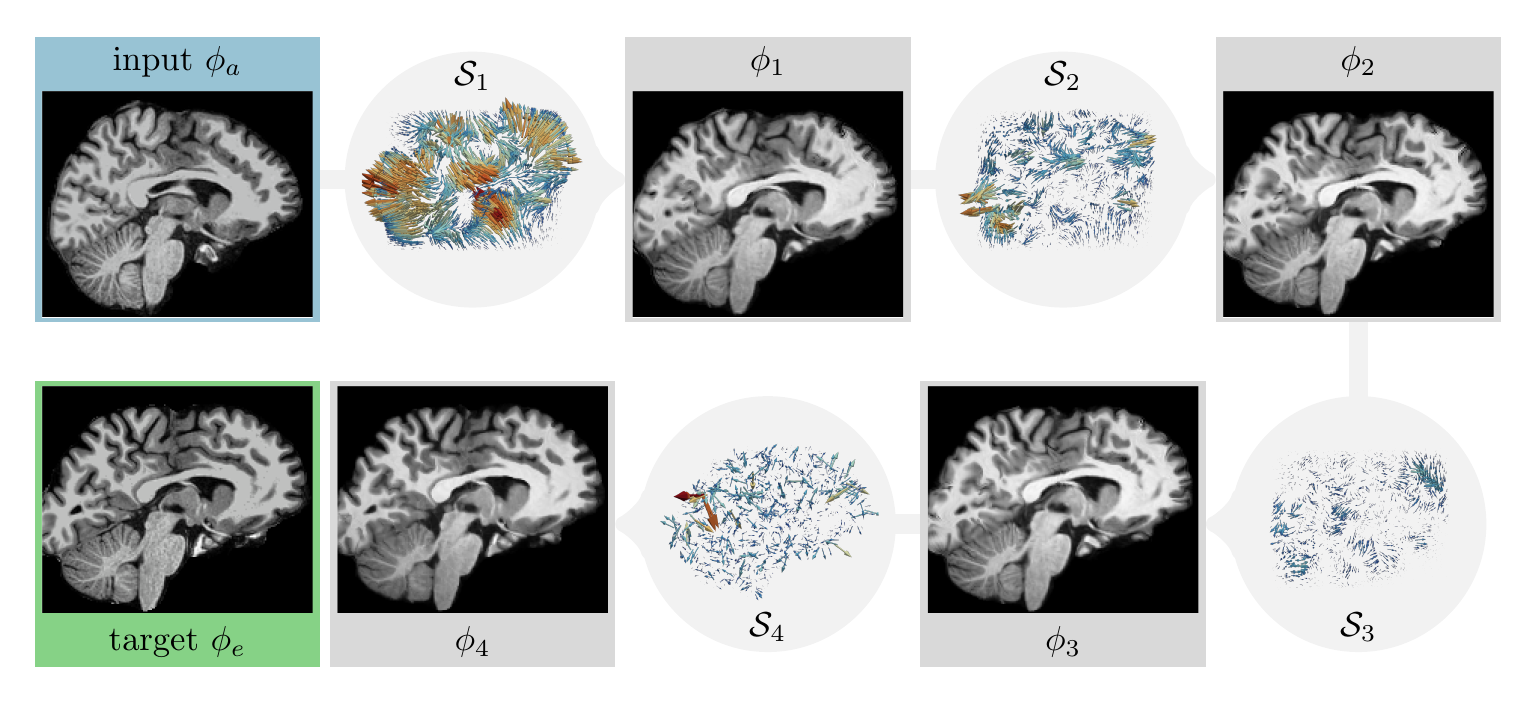}
\caption{
Illustration of our multistep registration approach. 
Starting from the affine pre--registered image $\phi_a$ (blue box), we apply a sequence of velocity field transformations to register the image to the target $\phi_e$ (green box).
Note that the vectors representing the velocity field are scaled equally in all illustrations, indicating that the first transformations learn the large distance deformations, while the last transformations learn to register small--scale details.
Further note that changing the hyperparameters $\alpha, \beta$ for $\mathcal{S}_4$ yields a velocity field with increased local variations as visible by the two large velocity vectors in the lower row.
\label{fig:multistep}
}
\end{figure}

We first optimize a series of three velocity registrations with distinct velocity fields $v_3$, $v_2$, $v_1$ for a total of 400 iterations of L-BFGS-B.
The hyperparameters to smoothen the control are set to $\alpha=0, \beta=1$.
In detail, we start from the affine registered image $\phi_a$ and optimize $v_1$ such that the discrepancy (\eqref{eq:data-discrepancy} and \eqref{eq:huber}) between the deformed image $\phi_1$ and the target image $\phi_e$ is reduced.
We then hold $v_1$ fixed and optimize $v_2$ such that the discrepancy between $\phi_2$ and $\phi_e$ is further reduced, and then repeat the procedure.
This reduces the $L^2$--discrepancy (which we only use for tracking the current progress of the optimization algorithm) between deformed image and target by 60 \% (relative to the affine pre--registration), cf. Fig. \ref{fig:optimization}. We next investigate whether our approach can improve upon the registration by changing the parameters $\alpha, \beta$ in \eqref{math:velocity-pre--processing}. 
We apply a fourth  registration with velocity field $v_4$ to the deformed image $\phi_3$ obtained with the procedure described above.
For this velocity field, we change from $\alpha=0, \beta=1$ to $\alpha=\beta=0.5$.
As can be seen from Fig. \ref{fig:optimization}, the mismatch between deformed image and target can be further reduced from 40 \% to 20 \% (relative to the affine pre--registration) after optimizing $v_4$ for 100 iterations. In Fig. \ref{fig:fs-comp-brain}, we depict the absolute difference between deformed images and target in color (heatmap) together with the target MRI (greyscale) in an exemplary axial slice. 

Additionally, we visualize a registration generated by the software package CLAIRE, see \cite{mang2019claire}. Unlike our current approach, this semi-Lagrangian framework is implemented in a highly efficient (GPU) manner and thus allows for automated hyperparameter tuning. It outperforms our three step registration. However, in Fig. \ref{fig:third} it can be observed that some notable error remains. Even with manual hyperparameter tuning we were not able to eliminate this error. Our four step procedure achieves better accuracy. 

\begin{remark}
    The multi--step deformation approach described above is equivalent to solving the PDE problem \eqref{eq::pde} with a time--dependent velocity over a time interval $[0, 4]$ where the velocity is piecewise constant in time,
    \begin{align*}
        v(\cdot, t) = v_i(\cdot), \quad t \in (i-1, i] , \quad i \in \{1,2,3,4\}.
    \end{align*}
Due to memory limitations, we optimize $v_1, \dots, v_4$ separately. This is in general not equivalent to optimizing over all velocities simultaneously. 
    \label{sec:remark-time-dependency}
\end{remark}

\begin{figure}[h]
\centering
\includegraphics[width=0.65\textwidth]{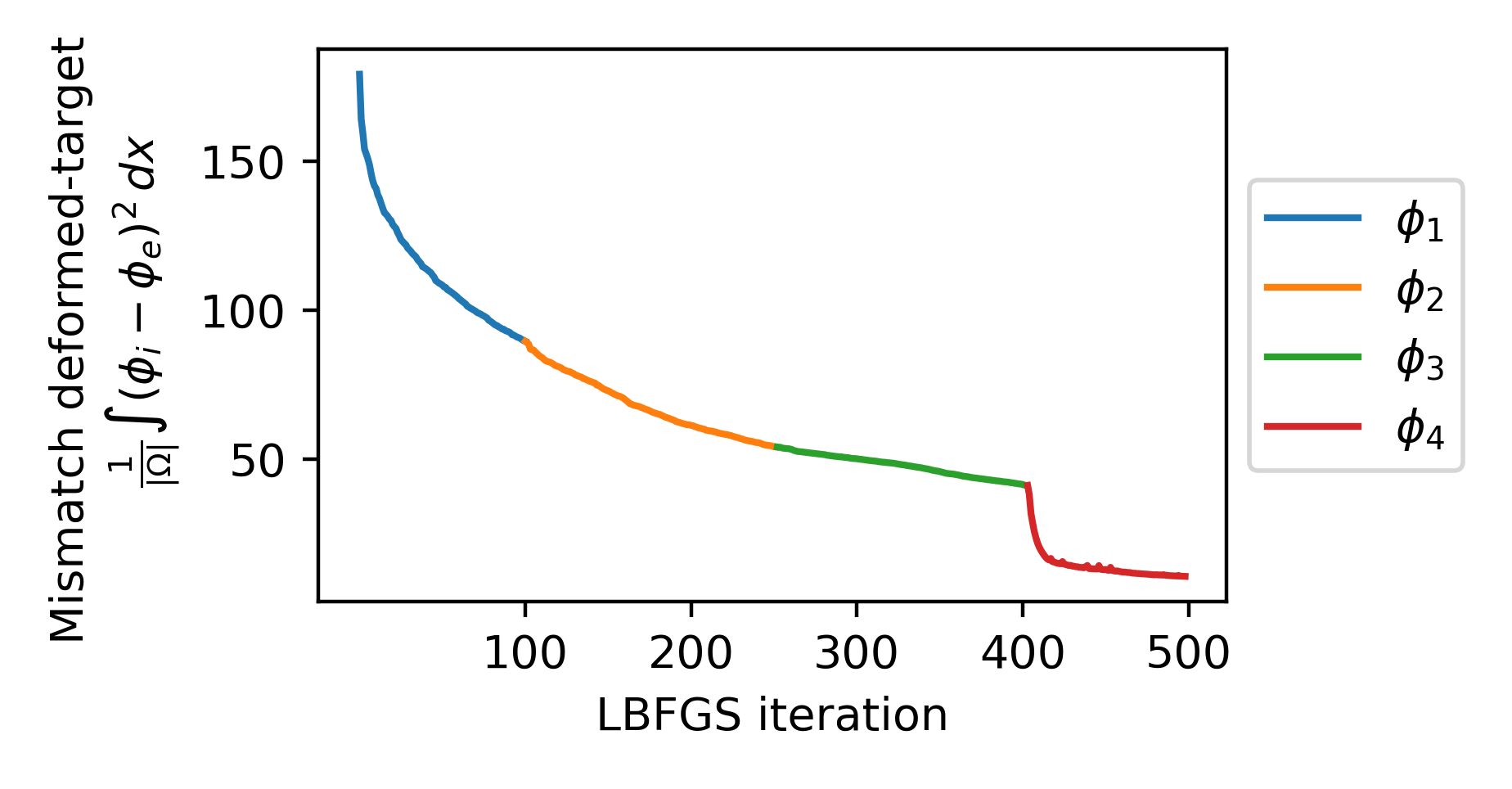}
\caption{
Reduction of the $L^2$-norm between deformed image and target during optimization of the velocity field deformations $\mathcal{S}_i$, $i=1,2,3,4$, with the L-BFGS-B algorithm.
$\phi_1,\phi_2,\phi_3$ are obtained with $\alpha=0, \beta=1$ and $\phi_4$ is obtained with $\alpha=0.5, \beta=0.5$.
\label{fig:optimization}
}
\end{figure}

In order to quantify the difference between our registration approach and FreeSurfer cvs--register, we apply the determined registration mapping to \texttt{norm.mgz} (another output of the \texttt{recon--all} command in FreeSurfer; \texttt{brain.mgz} is obtained by applying a normalization to \texttt{norm.mgz}) 
for Abby since this is the image that cvs--register works with. In the following, we denote the \texttt{norm.mgz} images of Abby and Ernie by $\hat \phi_a$ and $\hat \phi_e$, respectively.
Applying the transformation to a different input image ($\hat \phi_a$ instead of $\phi_a$ that we used for the registration) yields to numerical artefacts in this case, and we observe that around 600 voxels have unreasonable high intensity values $> 1000$ after registration. Hence, we quantify the difference between the registered image and \texttt{norm.mgz} from target Ernie using Tukey's biweight function \cite{tukey1960survey} given by
\begin{align*}
    \rho(x) = \begin{cases} \frac{c^2}{2} \left( 1- \left( 1 - \frac{x^2}{c^2}\right)^3 \right) \quad &\text{if } |x| \leq c,  \\
    \frac{c^2}{2} \quad &\text{otherwise.}
    \end{cases} 
\end{align*}
for different values of $c$ and report the result in Table \ref{tab:errornorm}.
We note that $c=255$ is the maximum voxel intensity value in \texttt{norm.mgz} for target Ernie.
We find that our registration based on four velocity fields yields a comparable Tukey difference as compared to FreeSurfer cvs--register for all parameters $c$. 
The three velocity field transform achieves a lower accuracy in the Tukey norm. Here, the transform obtained by (a hyperparameter tuned) CLAIRE \cite{mang2019claire} yields better results since it introduces less numerical artefacts. Note that hyperparameter tuning is not feasible with our current implementation. 
Fig. \ref{fig:fs-comparison} shows the registration error obtained with the nonlinear registration provided by FreeSurfer cvs--register \cite{postelnicu2008combined} (Surface and Volume based registration), and CLAIRE and our approaches applied to $\hat \phi_a$.

\begin{figure}[ht]
\centering
\begin{subfigure}{0.32\textwidth}
    \includegraphics[width=\textwidth]{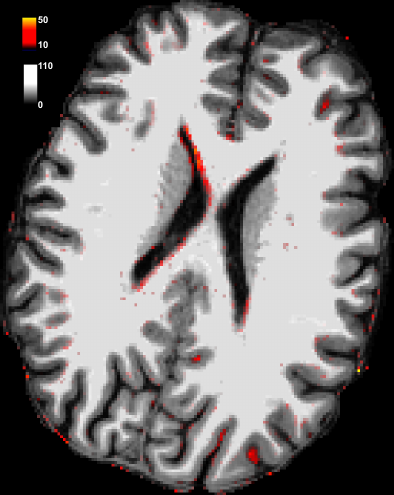}
    \caption{four step registration $\mathcal S_4 \circ \mathcal S_3 \circ \mathcal S_2 \circ \mathcal S_1 (\tilde \phi_a \circ \mathcal A^{-1}) - \tilde \phi_e$}
    \label{fig:first}
\end{subfigure}
\hfill
\begin{subfigure}{0.32\textwidth}
    \includegraphics[width=\textwidth]{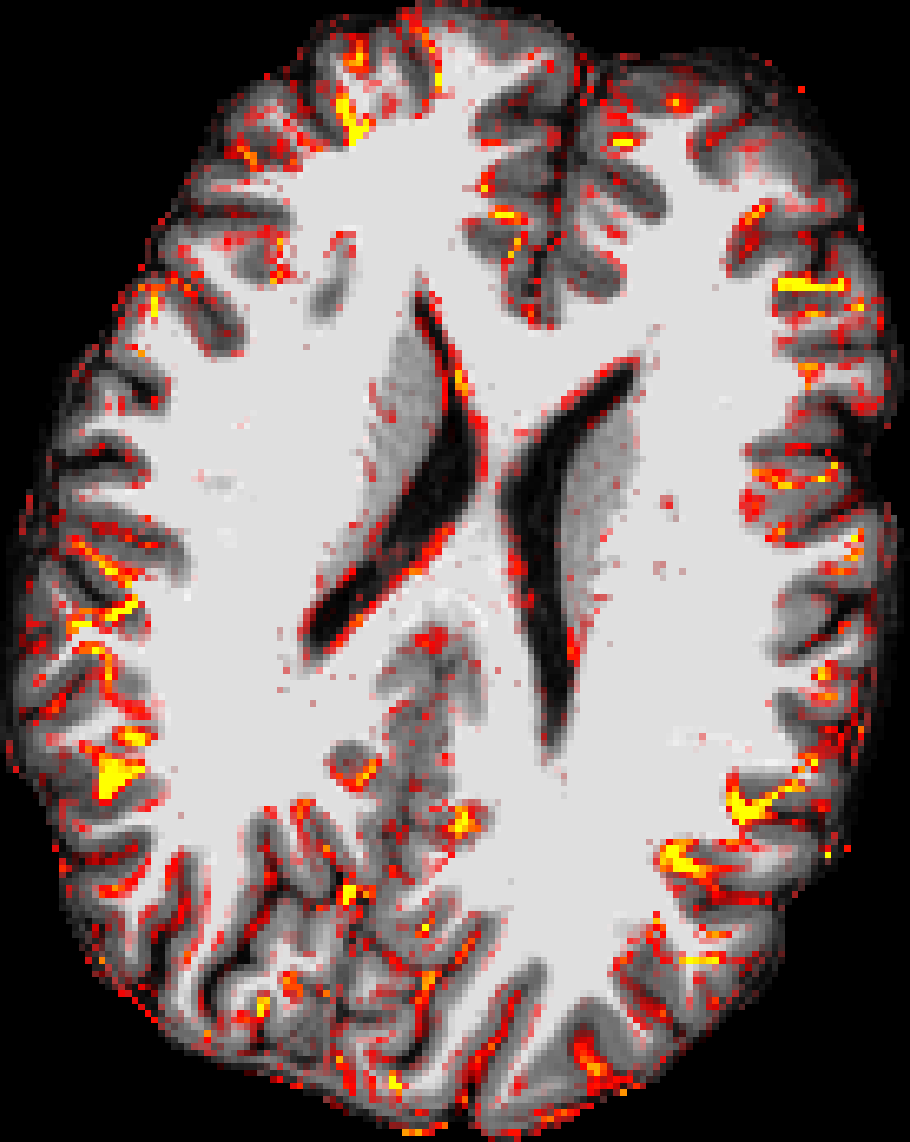}
    \caption{three step registration $\mathcal S_3 \circ \mathcal S_2 \circ \mathcal S_1 (\tilde \phi_a \circ \mathcal A^{-1})- \tilde \phi_e$}
    \label{fig:second}
\end{subfigure}
\hfill
\begin{subfigure}{0.32\textwidth}
    \includegraphics[width=\textwidth]{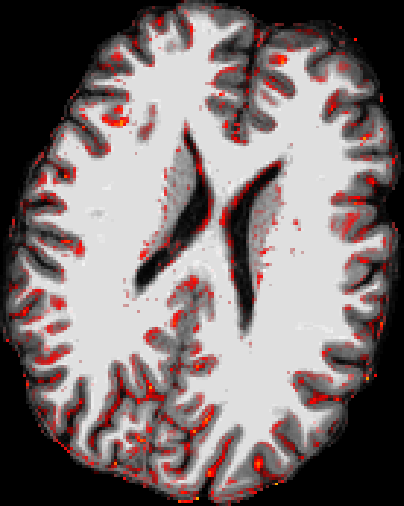}
    \caption{registration with CLAIRE \\ \ }
    \label{fig:third}
\end{subfigure}

\caption{
        The figure compares the results of our registration algorithm with three and four steps as well as CLAIRE \cite{mang2019claire}. 
        Background MR images: Slices from target image \enquote{Ernie} $\tilde \phi_e$ (\texttt{brain.mgz}, cf. Fig. \ref{fig::preprocessing}). 
        The heatmap shows the absolute difference between target $\tilde \phi_e$ and images registered with different procedures based on \enquote{Abby} $\tilde \phi_a$ (\texttt{brain.mgz}, cf. Fig. \ref{fig::preprocessing}).
}
        \label{fig:fs-comp-brain}

\end{figure}

\begin{figure}[ht]
\centering
\begin{subfigure}{0.24\textwidth}
    \includegraphics[width=\textwidth]{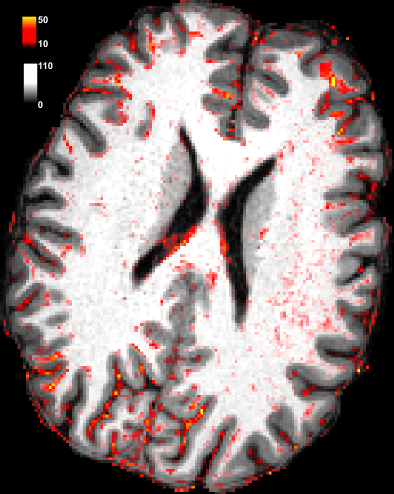}
    \caption{difference of FreeSurfer CVSregister (based on $\hat \phi_a$) and $\hat \phi_e$}
    \label{fig:norm_first}
\end{subfigure}
\hfill
\begin{subfigure}{0.24\textwidth}
    \includegraphics[width=\textwidth]{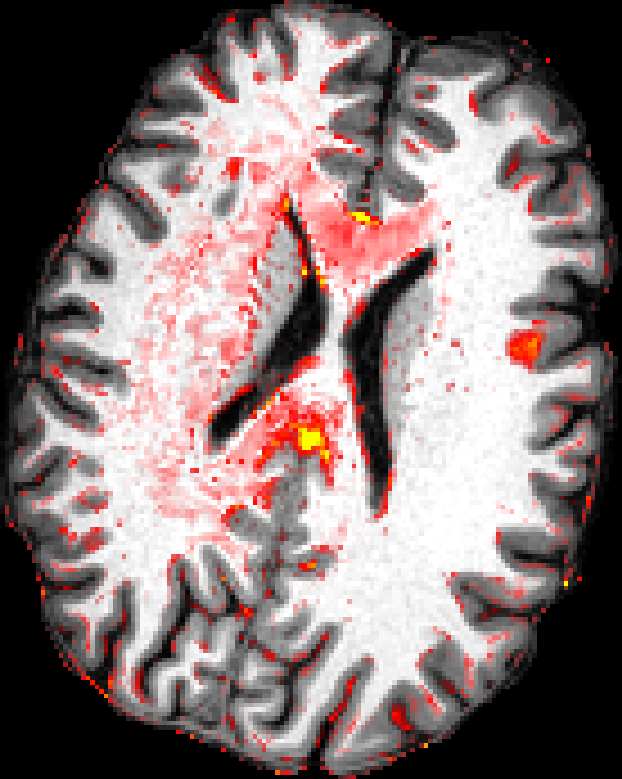}
    \caption{four step registration $\mathcal S_4 \circ \mathcal S_3 \circ \mathcal S_2 \circ \mathcal S_1 (\hat \phi_a \circ \mathcal A^{-1}) - \hat \phi_e$ \\}
    \label{fig:norm_second}
\end{subfigure}
\hfill
\begin{subfigure}{0.24\textwidth}
    \includegraphics[width=\textwidth]{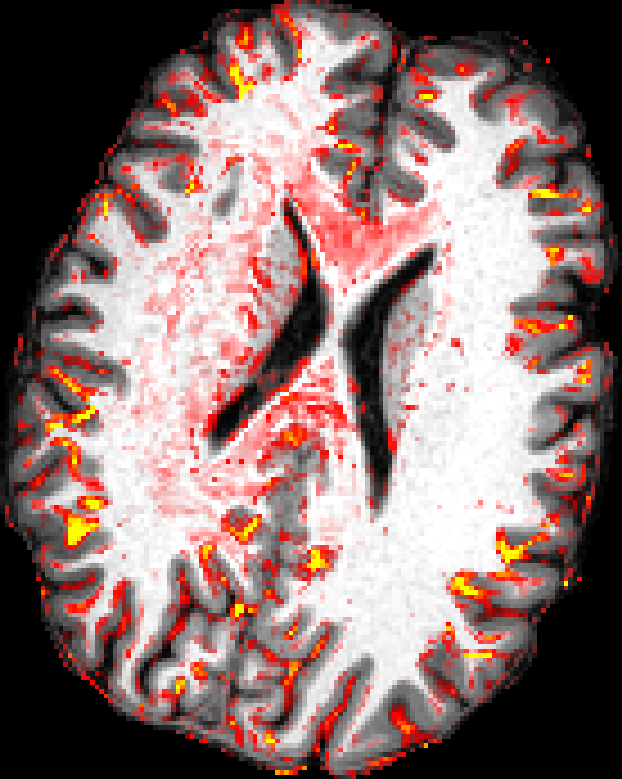}
    \caption{three step registration $\mathcal S_3 \circ \mathcal S_2 \circ \mathcal S_1 (\hat \phi_a \circ \mathcal A^{-1})- \hat \phi_e$ \\}
    \label{fig:norm_third}
\end{subfigure}
\hfill
\begin{subfigure}{0.24\textwidth}
    \includegraphics[width=\textwidth]{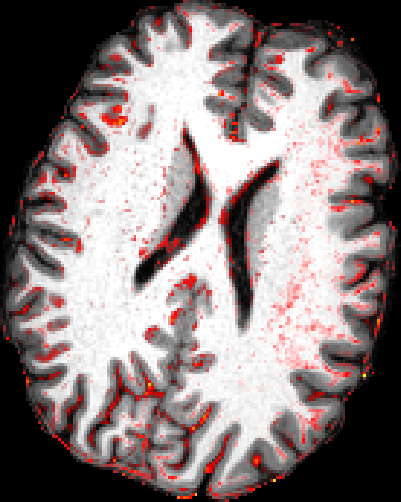}
    \caption{difference of CLAIRE (based on $\hat \phi_a$) and $\hat \phi_e$ \\}
    \label{fig:norm_forth}
\end{subfigure}
\caption{
        The figure compares the results of our registration algorithm to results obtained with FreeSurfer CVSregister \cite{postelnicu2008combined} as well as CLAIRE  \cite{mang2019claire}. 
        Background MR images: Slices from target image \enquote{Ernie} $\hat \phi_e$ (\texttt{norm.mgz}). 
        The heatmap shows the absolute difference between target $\hat \phi_e$ and images registered with different procedures based on \enquote{Abby} $\hat \phi_a$ (\texttt{norm.mgz}).
}
        \label{fig:fs-comparison}

\end{figure}


\begin{table}[]
\centering
\caption{Tukey error norm between registered image and norm.mgz for target Ernie for different values $c$ \label{tab:errornorm}}
\begin{tabular}{c|cccc}
$c$   & 3--velocity transform & 4--velocity transform & FreeSurfer cvs--register& CLAIRE \\
50  & 21.10      & 14.40      & 14.43  & 12.78 \\
110 & 28.85      & 17.65      & 18.23  & 15.30 \\
255 & 32.58      & 20.25      & 20.10  & 16.57 \\
\end{tabular}
\end{table}

\section{Mesh transformation}
\label{sec:meshtrafo}
In the following, we describe how the vertex coordinates of the input mesh from \enquote{Abby} are deformed to obtain a new mesh for \enquote{Ernie}, see Fig. \ref{fig::meshtransformationpipeline}.

\begin{figure}[h]
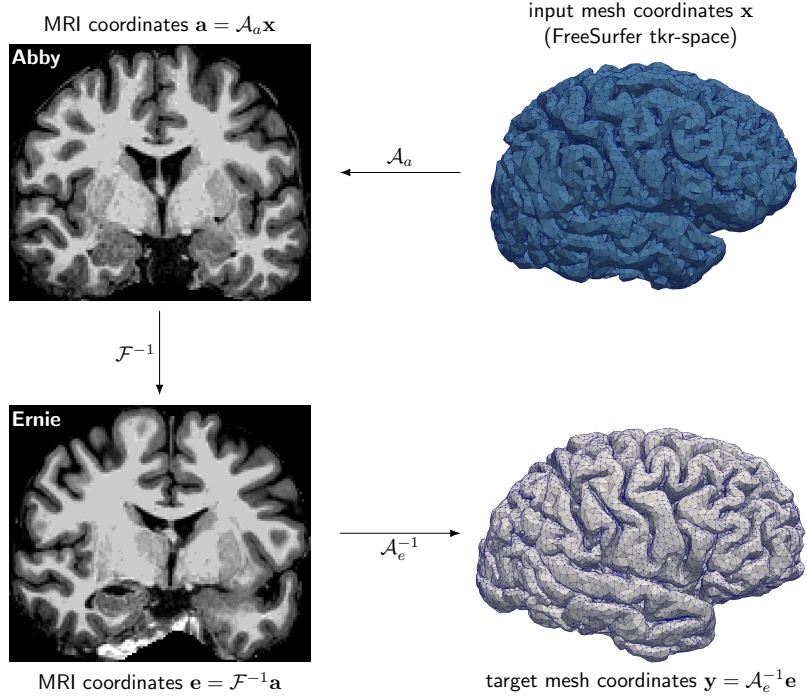

\centering
\scalebox{0.8}{
\begin{tikzpicture}[font=\sf]
\draw[-{Latex[]}] (0, 3.7) -- (0, 2.3);
\draw[{Latex[]-}] (3, 6) -- (5, 6);
\draw[{-Latex[]}] (3, 0) -- (5, 0);
\node[align=center] at (0, 8.45) { MRI coordinates $\mathbf a = \mathcal A_a \mathbf x$};
\node[align=center] at (0, -2.45) { MRI coordinates $\mathbf e = \mathcal F^{-1} \mathbf a$};
\node[align=center] at (8, 8.45) { input mesh coordinates $\mathbf x$\\ (FreeSurfer tkr-space)};
\node[align=center] at (8, -2.45) { target mesh coordinates $\mathbf y = \mathcal A_e^{-1} \mathbf e$};
\node[align=center] at (4, 6.25) {$\mathcal A_a $};
\node[align=center] at (4, -0.25) {$\mathcal A_e^{-1} $};
\node[align=center] at (-0.4, 3) {$\mathcal F^{-1} $};
\node[inner sep=0pt] at (0.,0)    {\includegraphics[width=5cm, ]{figures/ernie_nomesh.png}};
\node[inner sep=0pt] at (0.,6)    {\includegraphics[width=5cm, trim={0.25cm 0.25cm 0.25cm 0.25cm}, clip]{figures/abby_nomesh.png}};
\node[inner sep=0pt] at (8.0, 0.1)    {\includegraphics[width=7cm, trim={15cm 0cm 15cm 0cm}, clip]{figures/ernie2.png}};
\node[inner sep=0pt] at (8.0, 6.1)    {\includegraphics[width=7cm, trim={15cm 0cm 15cm 0cm}, clip]{figures/abby2.png}};
\node[align=left, color=white] at (-2.05, 1.9) {\textbf{Ernie}};
\node[align=left, color=white] at (-2.05, 7.9) {\textbf{Abby}};
\end{tikzpicture}
}
\caption{Mesh transformation pipeline illustration the different coordinate systems involved.}
\label{fig::meshtransformationpipeline}
\end{figure}

\subsection{Mesh transformation pipeline}

We use SVMTK \cite{svmtk} as described in \cite{mardal2022mathematical} to create the reference mesh for \enquote{Abby}.
The meshing pipeline in SVMTK uses the surface files generated by the FreeSurfer \texttt{recon--all} command, which automatically generates cortical surface meshes and brain segmentations from MRI. 
For the numerical results presented in this paper, we assume that \texttt{recon--all} has been run successfully on subject $a$, and use the automatic parcellation output file \verb|aseg.mgz| as well as the left pial surface (the left hemisphere brain surface) file \verb|lh.pial| this command generates.
We distinguish between $\mathbf{x} \in \mathbb{R}^3$ for the coordinate of a vertex on a FreeSurfer surface and the corresponding coordinate $\mathbf{a}\in [0, 256]^3$ in the MRI voxel space as
\begin{align*}
    \mathbf{a} = \mathcal{A}_a\mathbf{x} = A_a \mathbf{x} + b_a.
\end{align*}
Here, $A_a \in \mathbb{R}^{3\times 3}$, $b_a\in \mathbb{R}^{3}$, parametrize the affine transform $ \mathcal{A}_a$ between coordinates defined in FreeSurfer-tkr space and voxel space, cf., e.g., \cite[Section 4.4]{mardal2022mathematical} for details on the coordinate systems involved in the mesh generation.
We remark that his transformation is subject--specific and is can be read from the files produced by performing surface generation using the FreeSurfer \texttt{recon--all} function.

We define $\mathcal F$ as the mapping that registers the input image to the target, i.e.,
\begin{align*}
 \mathcal F := \tau \circ \mathcal A
\end{align*}
with
\begin{align*}
 \tau := \mathcal{V}_N \circ \dots \circ \mathcal{V}_1 ,
\end{align*}
where $N \in \mathbb N$, $\mathcal A$ denotes the affine registration of the input and target image, $v_i$ denotes the optimized velocities obtained in section \ref{subsubsec:imagereg}, and $\mathcal V_i$ is defined by
 \begin{align*}
    {\mathcal V}_i : \Omega \to \Omega, \quad x \mapsto X(x, T),
 \end{align*}
 with $X$ given as the solution of
\begin{align*}
\begin{split}
\partial_t X + \nabla X^\top v_i = 0 \quad &\text{on }Q_T, \\
X(0) = x \quad & \text{in } \Omega.
\end{split}
\end{align*}
Moreover, we relate
\begin{align*}
\mathbf{e} = \mathcal F^{-1} \mathbf{a} = (\mathcal{A}^{-1} \circ \mathcal V_1^{-1} \circ \dots  \circ \mathcal V_N^{-1}) \mathbf{a}
\end{align*}
and compute $\mathcal F^{-1}$ by solving the CG1 discretized transport equation with the negative optimized velocity fields and the coordinates as initial condition.

Finally, in order to compare the meshes generated by our approach to surfaces generated by FreeSurfer \texttt{recon--all}, we transform the deformed mesh coordinates $\mathbf{e}$ into the FreeSurfer surface coordinate space by
\begin{align*}
    \mathbf{y} = \mathcal{A}_e^{-1} \mathbf{e},
\end{align*}
where $ \mathcal{A}_e$ is the affine transform between voxel coordinates and FreeSurfer surface coordinates for the target image.

\subsection{Numerical results} \label{subsubsec:meshreg}

Using the registration deformation obtained above, we test the deformation of two different meshes, one for the ventricular system and one for the left hemisphere. 
The ventricular system is chosen because this is an example, where manual mesh processing is required and we want to demonstrate that our registration can be used to create a new ventricular system mesh for the target without manual corrections.
The left hemisphere mesh is chosen to show the limitations of our algorithm in the current formulation.

We create a ventricular system mesh for the template \enquote{Abby} by utilizing the automatic segmentation image \verb|aseg.mgz| provided by FreeSurfer and mark the cerebral aqueduct manually (cf. Fig. \ref{fig:abby-aq}A). 
Using the processing steps described in \cite{mardal2022mathematical}, we create a volume mesh for the ventricular system (cf. Fig. \ref{fig:abby-aq}C).
The mesh for the left hemisphere is generated using the post--processing steps in \cite{mardal2022mathematical} from the surface \texttt{lh.pial} generated by FreeSurfer's \texttt{recon--all} command. 

First, we apply the transformations $\mathcal{A}$ (affine pre--registration), $\mathcal V_1\circ \mathcal{A}$ (affine and one velocity field deformation) as well as $\mathcal V_3 \circ \mathcal V_2 \circ \mathcal V_1 \circ \mathcal{A}$ (affine and three velocity field deformations) to the ventricular system mesh and compare the boundaries of these meshes in Fig. \ref{fig:abby-ventricles}.
It can be seen that the affine registered mesh does not match the target \enquote{Ernie} at all (red contours in Fig. \ref{fig:abby-ventricles}). 
Transport with one velocity field gives visibly better agreement (orange contours in Fig. \ref{fig:abby-ventricles}).
To resolve finer details and obtain meshes that are accurate to within a few voxels, two more velocity fields are needed (green contours and blue arrows in Fig. \ref{fig:abby-ventricles}).

Figure \ref{fig:abby-aq}D shows the deformed mesh. It can be seen that the mesh transformation with $\mathcal V_3\circ\mathcal V_2\circ\mathcal V_1\circ \mathcal{A}$ yields a mesh with similar smooth surface as the input mesh.  
Hence, one does not need to manually label the cerebral aqueduct, which is missing in the FreeSurfer automatic segmentation of \enquote{Ernie} (Fig. \ref{fig:abby-aq}B).

\begin{figure}[h]
\centering
\scalebox{0.8}{
\begin{tikzpicture}[font=\sf]
\node[inner sep=0pt] at (0.,0)    {\includegraphics[width=6cm, trim={5cm 0cm 6.5cm 0cm}, clip]{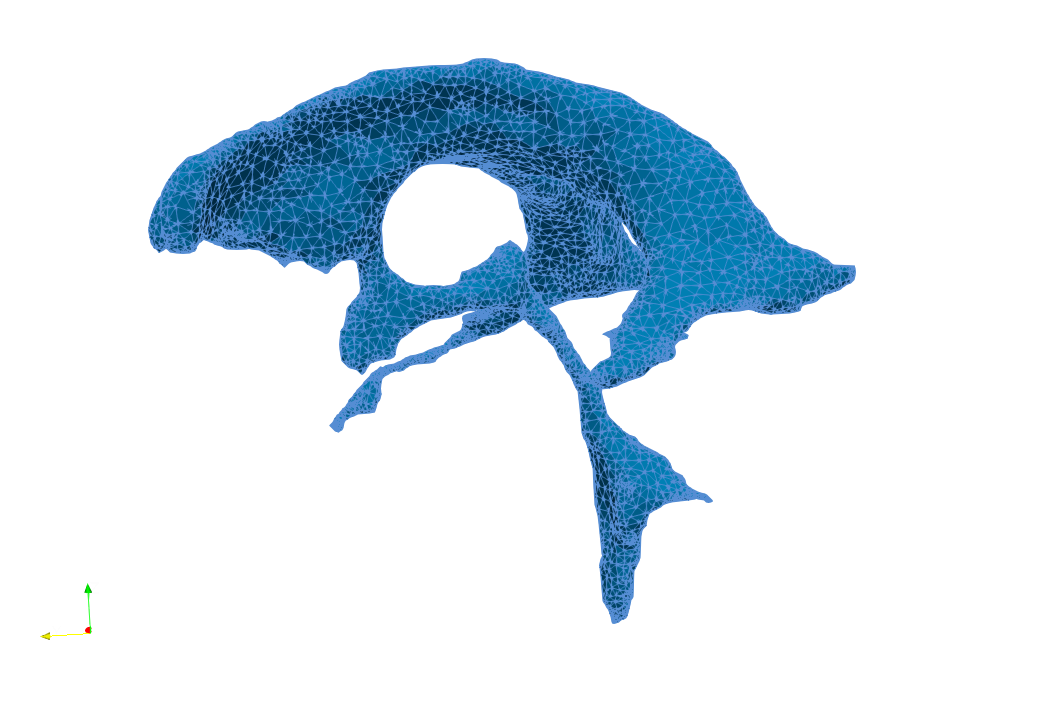}};
\node[inner sep=0pt] at (0.,7)    {\includegraphics[width=7cm, trim={0cm 0cm 4cm 0cm}, clip]{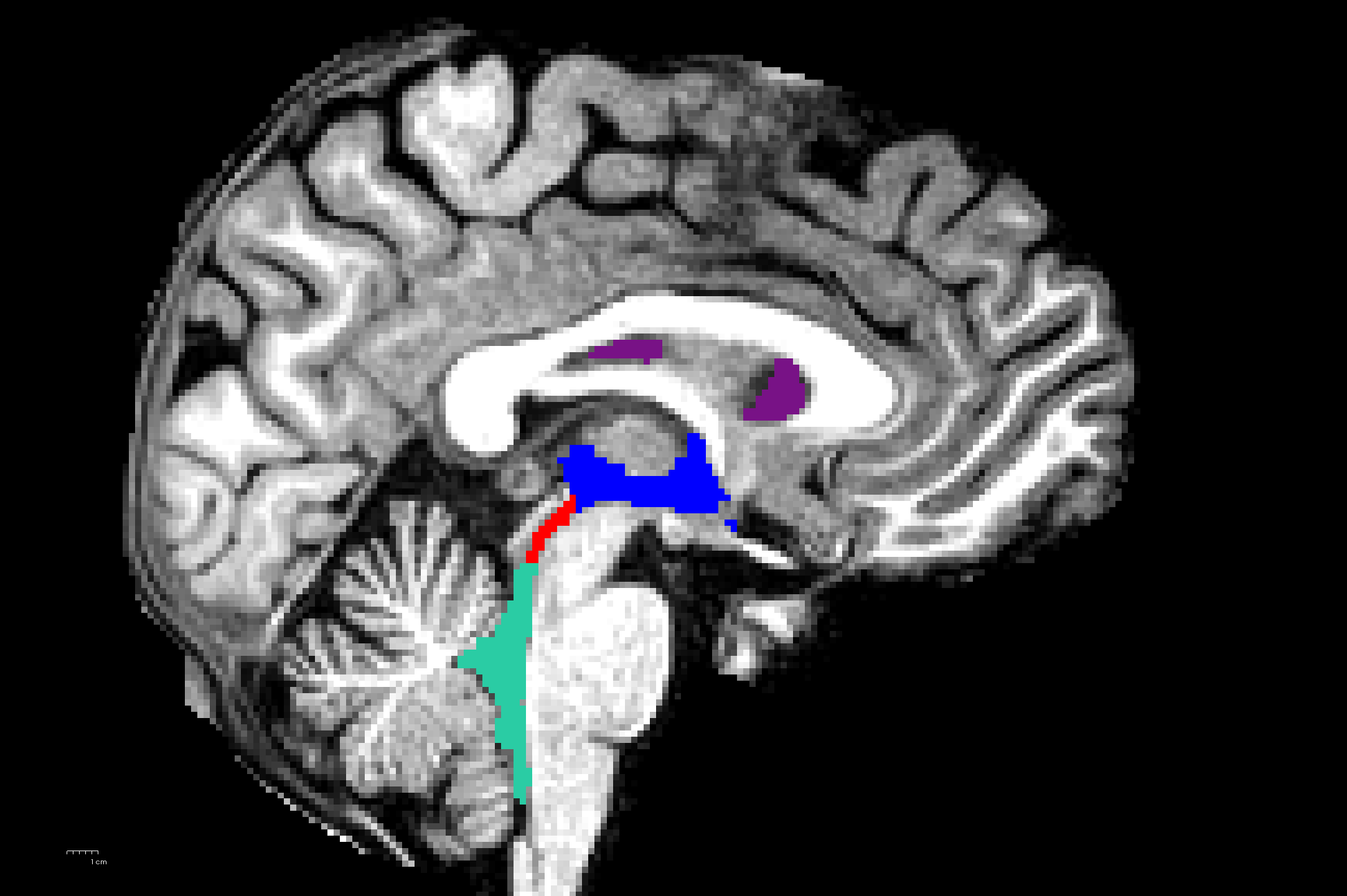}};
\node[inner sep=0pt] at (8.0, 0)    {\includegraphics[width=6cm, trim={4cm 0cm 5cm 0cm}, clip]{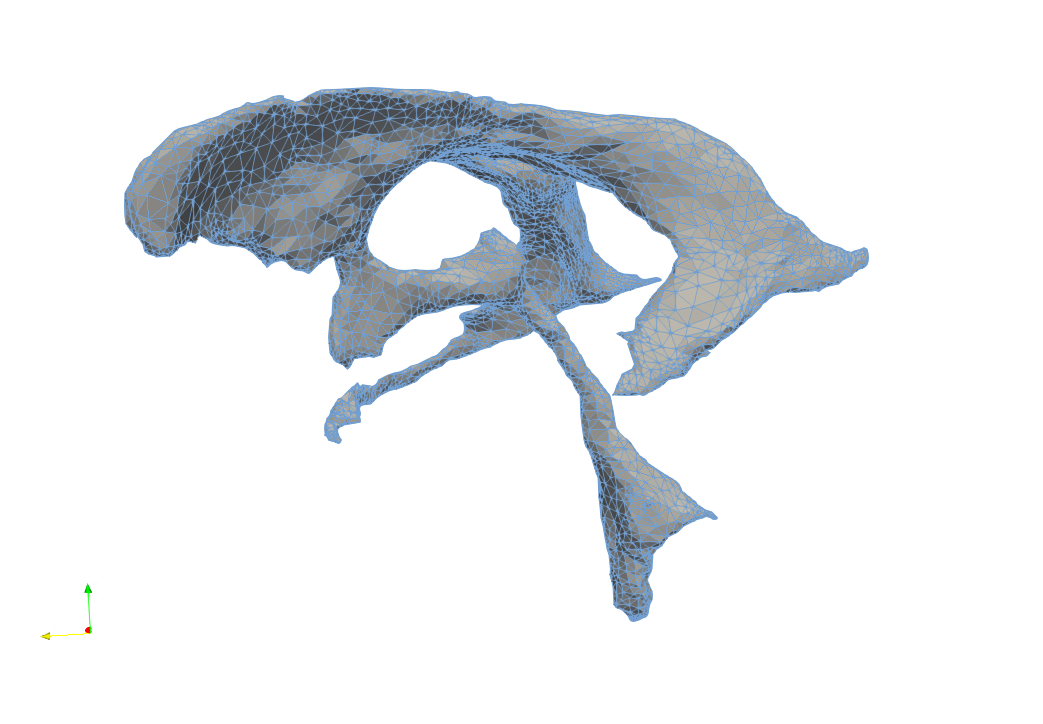}};
\node[inner sep=0pt] at (8.0, 7)    {\includegraphics[width=7cm, trim={4cm 3.5cm 5cm 0cm}, clip]{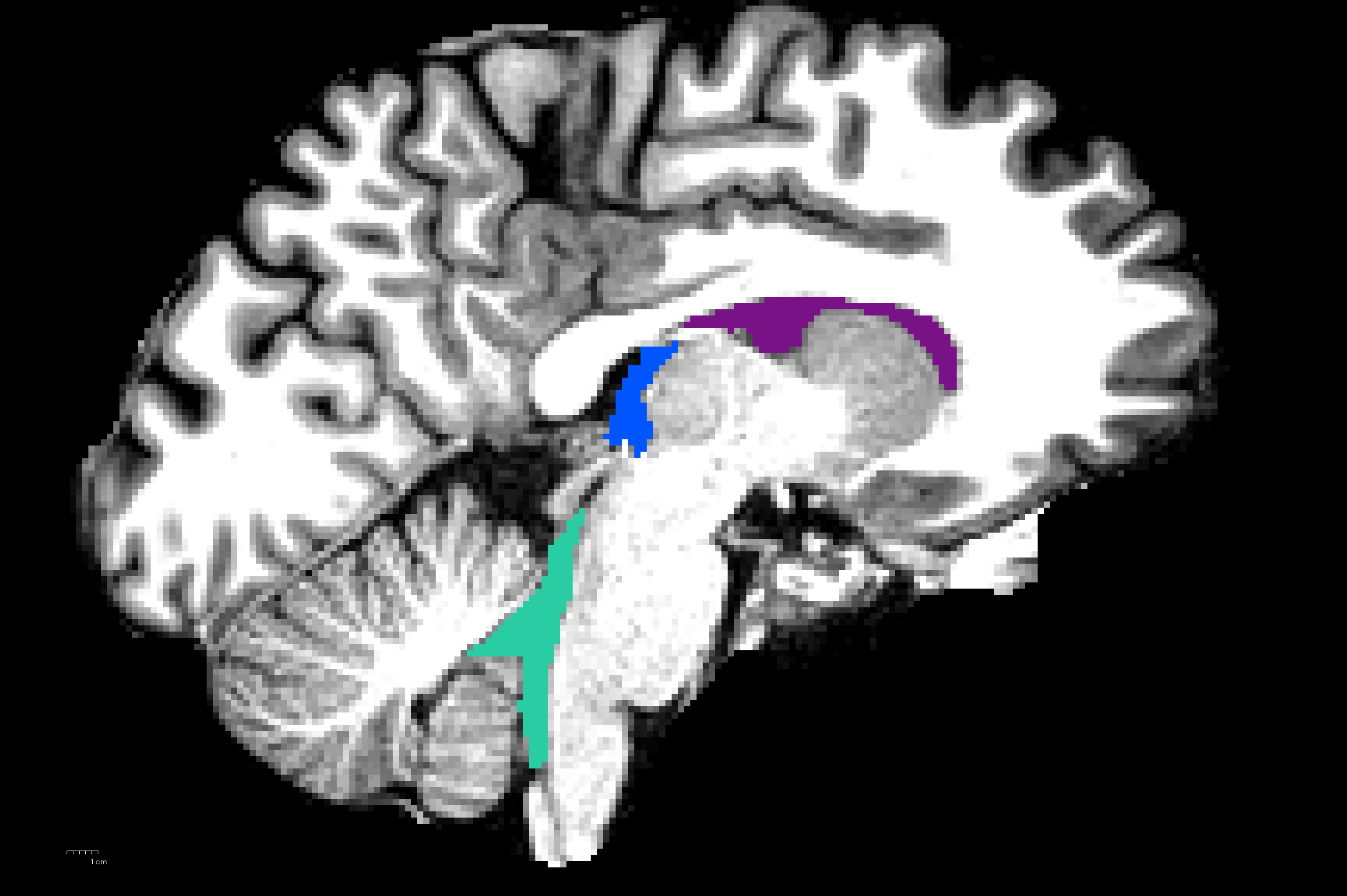}};
\draw[{-Latex[]}] (2, 0) -- (6, 0);
\node[align=center] at (0, -2.45) {input mesh coordinates $\mathbf x$};
\node[align=center] at (8, -2.45) {target mesh coordinates $\mathbf y$};
\node[align=center] at (4, -0.25) {$\mathcal A_e^{-1} \circ \mathcal F^{-1} \circ \mathcal A_a $};
\node[align=left, color=white] at (-3.25, 9.2) {\textbf{A}};
\node[align=left, color=white] at (4.75, 9.2) {\textbf{B}};
\node[align=left] at (-3.25, 2.2) {\textbf{C}};
\node[align=left] at (4.75, 2.2) {\textbf{D}};
\draw[-{Latex[length=0.5cm, width=0.5cm]}, line width = 0.2cm, gray!90] (0, 4.25) -- (0, 2.75);
\node[align=center] at (-1.5, 3.5) { FreeSurfer \& \\ SVMTK \& \\ manual work};
\end{tikzpicture}
}
\caption{
A-B: Sagittal slice of \enquote{Abby} and \enquote{Ernie} with the FreeSurfer automatic segmentation for fourth ventricle (cyan), third ventricle (blue) and lateral ventricles (violet). 
After manually marking the cerebral aqueduct in \enquote{Abby} (displayed in red in A), we generate a mesh for the ventricular system of \enquote{Abby} (C).
The same nonlinear registration $\mathcal F$ that registers the MR image of \enquote{Abby} to \enquote{Ernie} can be used to obtain a mesh for the ventricular system of \enquote{Ernie} (D).
\label{fig:abby-aq}
}
\end{figure}
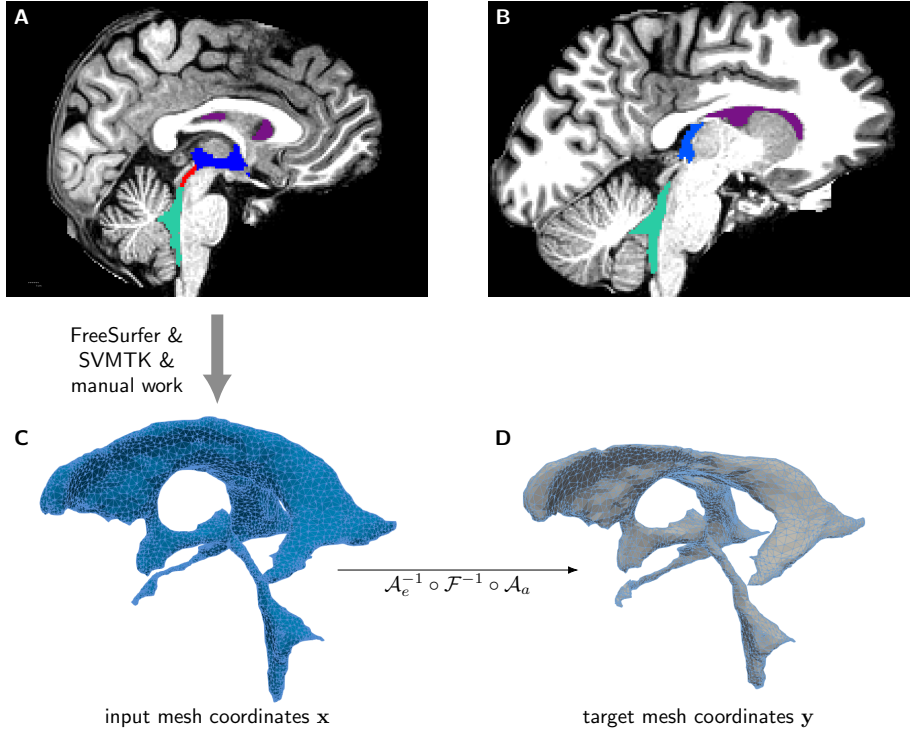

\begin{figure}[h]
\centering
\includegraphics[width=0.32\textwidth, trim={7cm 0cm 2cm 0cm}, clip]{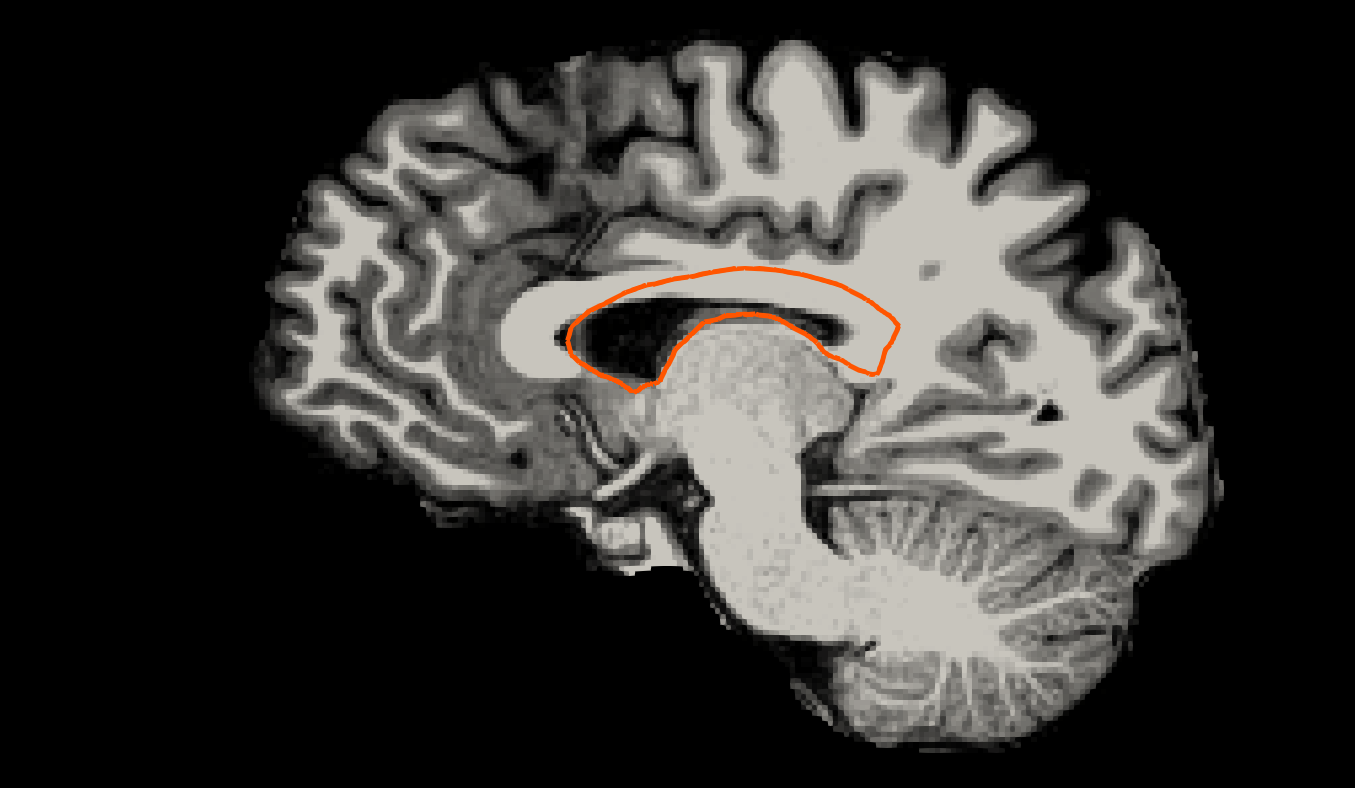}
\includegraphics[width=0.32\textwidth, trim={7cm 0cm 2cm 0cm}, clip]{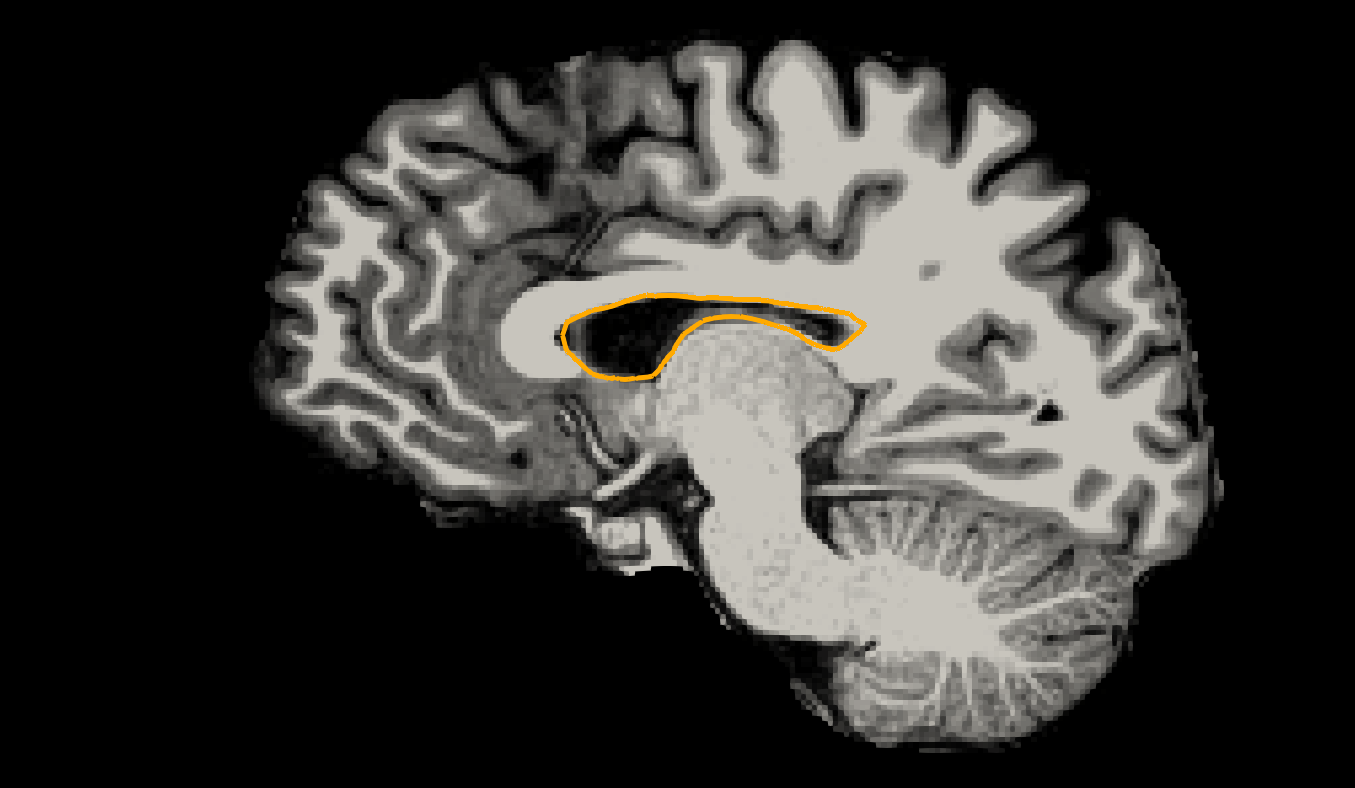}
\includegraphics[width=0.32\textwidth, trim={7cm 0cm 2cm 0cm}, clip]{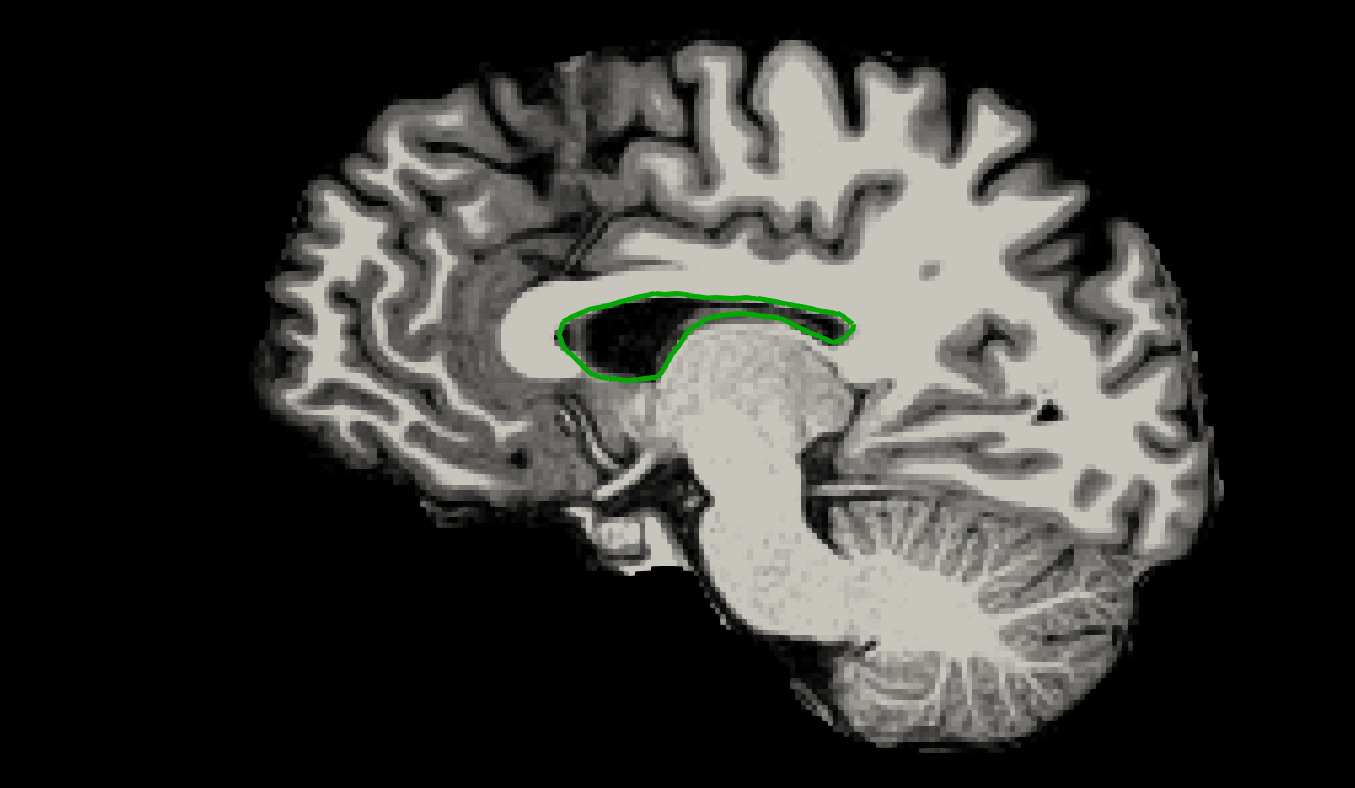} \\
\includegraphics[width=0.32\textwidth, trim={6.5cm 2cm 4.5cm 0cm}, clip]{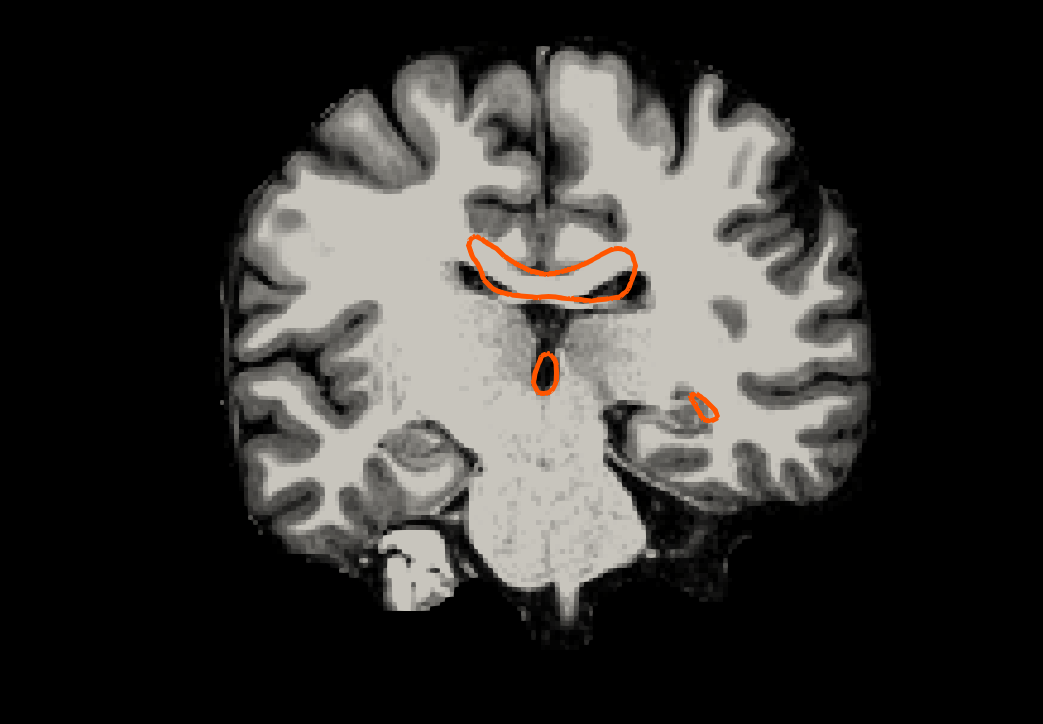}
\includegraphics[width=0.32\textwidth, trim={6.5cm 2cm 4.5cm 0cm}, clip]{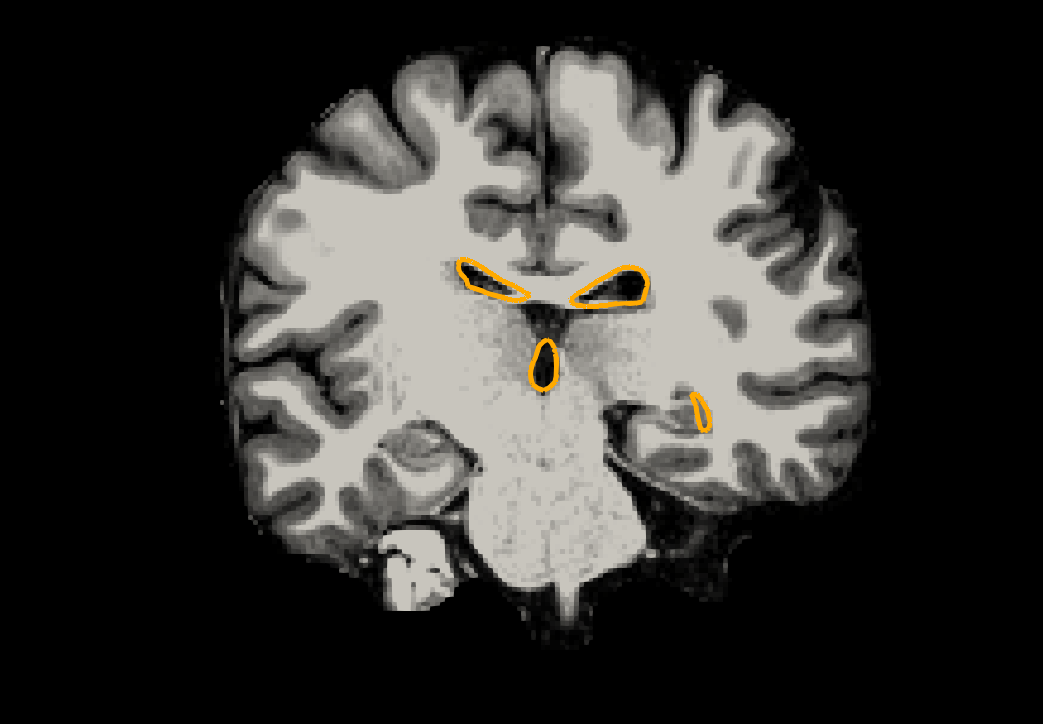}
\includegraphics[width=0.32\textwidth, trim={6.5cm 2cm 4.5cm 0cm}, clip]{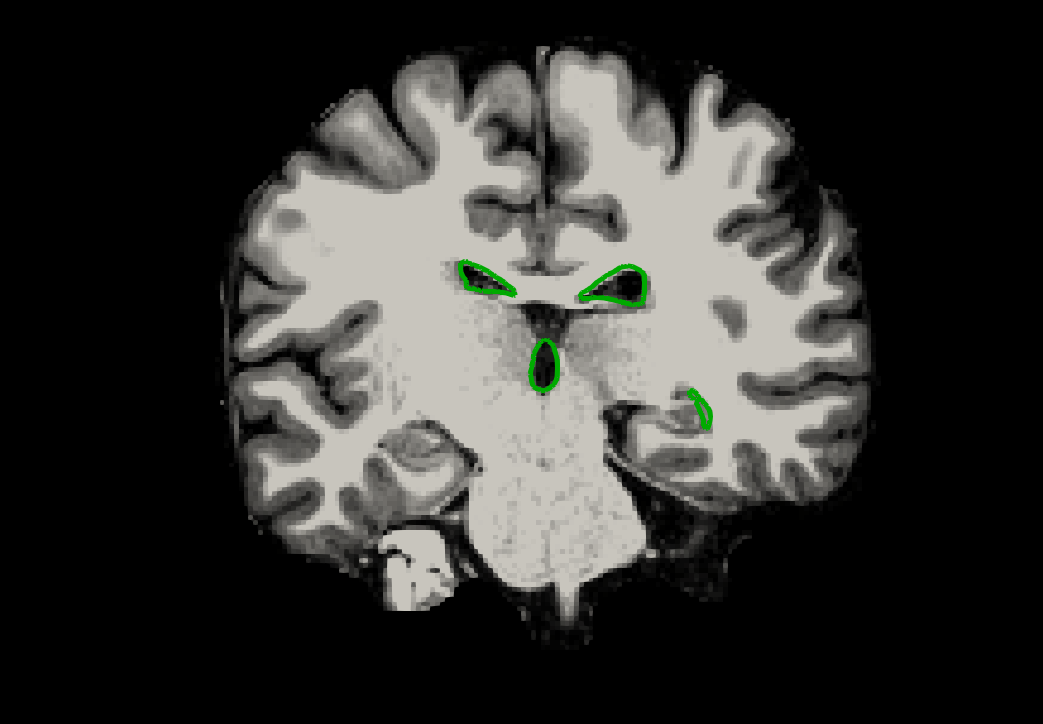} \\
\includegraphics[width=0.32\textwidth, trim={9cm 2cm 9cm 0cm}, clip]{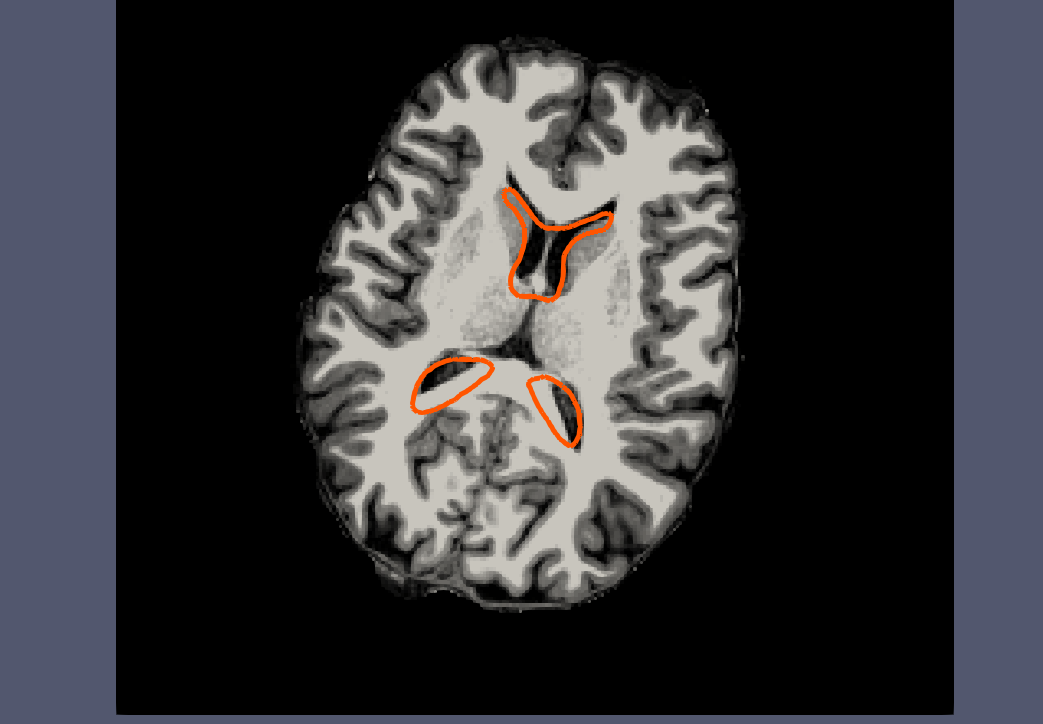}
\includegraphics[width=0.32\textwidth, trim={9cm 2cm 9cm 0cm}, clip]{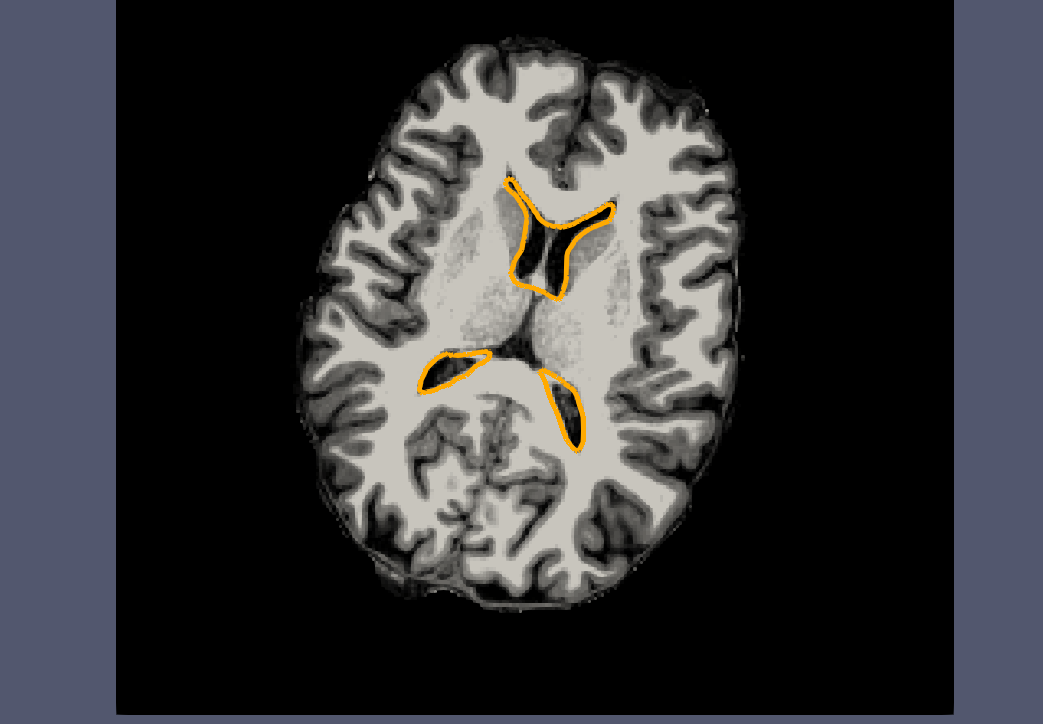}
\includegraphics[width=0.32\textwidth, trim={9cm 2cm 9cm 0cm}, clip]{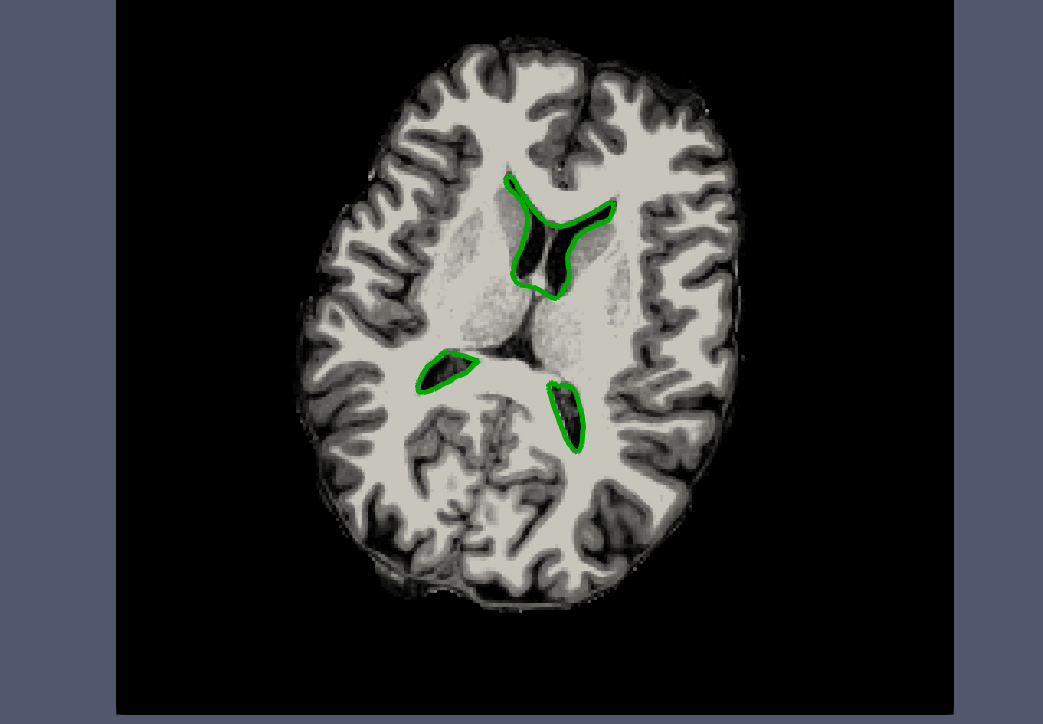}
\caption{
Background MR images: Slices from target image \enquote{Ernie}. Colored lines: surfaces of deformed ventricle mesh from input \enquote{Abby}.
Orange (left column): Mesh deformed with affine deformation $\mathcal{A}$, Yellow (mid column): Affine and one velocity field deformation $\mathcal V_1 \circ \mathcal{A}$, Green (right column): Affine and three velocity field deformations $\mathcal V_3 \circ \mathcal V_2 \circ \mathcal V_1 \circ \mathcal{A}$. 
\label{fig:abby-ventricles}
}
\end{figure}

We also test the application of all four velocity field registrations ($\mathcal V_4 \circ \mathcal V_3\circ\mathcal V_2\circ\mathcal V_1\circ \mathcal{A}$) to the ventricular system mesh.
A slice of the surface of the resulting deformed mesh is shown in blue in Fig. \ref{fig:mesh-comparison}.
It can be seen that the additional deformation does not improve upon the match to the ventricle. 
This was to be expected, since Fig. \ref{fig:fs-comparison} shows that three velocity field deformations already register the ventricular system well to the target image.
However, it can be seen that the quality of the deformed mesh decreases in the sense that the mesh surface becomes jagged. 
This is due to the fact that $\mathcal V_4$ was obtained with $\alpha=\beta=1/2$.
The velocity pre--processing step \eqref{math:velocity-pre--processing} hence yields a velocity field with finer features than in $\mathcal V_3,\mathcal V_2,\mathcal V_1$ obtained with $\alpha=0,\beta=1$.
To mitigate this issue, a post--processing step (smoothening and isotropic remeshing as described in \cite{mardal2022mathematical}) can be applied to the mesh surface. 
This yields a mesh very similar to the mesh obtained with three velocity fields shown in green in Fig. \ref{fig:fs-comparison} (post--processed mesh not shown).

\begin{figure}[h]
\centering
\includegraphics[width=0.7\textwidth]{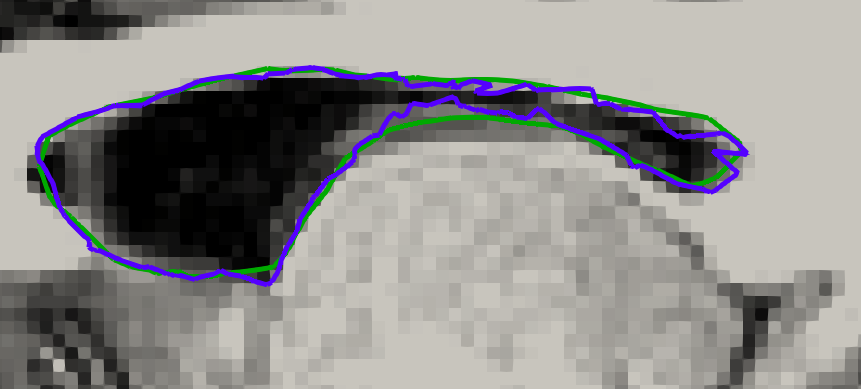}
\caption{
Surface of the deformed mesh obtained with three velocity fields (green) and an additional, less smooth velocity field (blue).
\label{fig:mesh-comparison}
}
\end{figure}

\section{Discussion and Outlook}
We presented a pipeline based image registration in order to obtain subject specific meshes. As a proof of concept, we focused on the ventricular system including the cerebral aqueduct. We are able to obtain a volume mesh that matches the ventricular system of the target subject from a manually crafted mesh matching the input subject. We also pointed out the limitations of the algorithm in its current formulation.
Adapting parameters required to obtain more accurate registration comes at the cost of introducing spikes when deforming the input mesh and instabilities when applying the transport equation to different initial data. Here, more sophisticated choices for the set of admissible velocity fields and hyperparameter tuning is expected to increase the performance of our algorithm.

Our pipeline consists of several sub-tasks, in which several heuristic choices occur (e.g. the choice of the normalization in the pre-processing and the parameterization of the velocity field in modelling the optimization problem). Moreover, our motivation of the upwind DG scheme points out weaknesses of the chosen numerical upwind DG scheme which also become apparent in the numerical results. Hence, improving our approach requires further improvements in the corresponding sub-tasks.

Our numerical implementation is based on the Python library FEniCS. 
This implementation does currently not make use of the full potential of the DG scheme to be parallelized more efficiently than, e.g., CG schemes. Hence, there is the need for a high performance implementation of the presented algorithm. This can, e.g., be realized by translating it into a neural network based approach based on, e.g., MeshGraphNets \cite{pfaff2020learning}.

\appendix

\section*{Acknowledgments}
This research has been supported by the Research Council of Norway under the grants 300305 and 320753 and the German Academic Exchange Service (DAAD) under grant 57570343 of the program ``Programm des projektbezogenen Personenaustauschs Norwegen 2021-2023''. This work has also been supported in part by the German Research Foundation (DFG) grant 314150341 of the priority program SPP1962. The computations presented in this work have been performed on the Experimental Infrastructure for Exploration of Exascale Computing (eX3), which is financially supported by the Research Council of Norway under contract 270053. The work was initiated during a workshop at the Simula cabin in Geilo, Norway. The authors wish to explicitly thank all of the respective institutions.

The authors also wish to thank Axel Kröner, Simon Funke, Rami Masri and Kathrin Völkner for their discussions and feedback on writing the grant proposal and manuscript, and during the workshop in Geilo, respectively.

\bibliographystyle{siamplain}
\bibliography{sample}

\end{document}